\documentclass[onefignum,onetabnum]{siamart220329}

\usepackage{amsmath,amsfonts,amssymb}
\usepackage{graphicx}
\usepackage{xcolor}
\usepackage{algorithm}
\usepackage{algorithmic}
\usepackage{array}
\usepackage{booktabs}
\usepackage{url}
\usepackage{hyperref}
\usepackage{microtype}
\usepackage{multirow}

\usepackage{subcaption}
\usepackage{tikz}
\usetikzlibrary{positioning,shapes,arrows}

\newcommand{\et}{\textit{et al.}}

\theoremstyle{remark}
\newtheorem{remark}{Remark}

\title{A Tensor Train Approach for  Deterministic Arithmetic Operations on Discrete Representations of Probability Distributions\thanks{Version of \today.}}

\author{Gerhard Kirsten\thanks{Signaloid.}
\and Bilgesu Bilgin\footnotemark[2] \and Janith Petangoda\footnotemark[2]
\and Phillip Stanley-Marbell\footnotemark[2] \thanks{Department of Engineering, University of Cambridge.}}

\headers{Tensor train distributional arithmetic}{G. Kirsten, B. Bilgin, J. Petangoda and P. Stanley-Marbell}

\begin{document}

\maketitle

\begin{abstract}
Computing with discrete representations of high-dimensional probability distributions is fundamental to uncertainty quantification, Bayesian inference, and stochastic modeling. However, storing and manipulating such distributions suffers from the curse of dimensionality, as memory and computational costs grow exponentially with the number of input distributions. Typical workarounds using Monte Carlo methods with thousands to billions of samples incur high computational costs and produce inconsistent results due to stochasticity. We present an efficient tensor train method for performing exact arithmetic operations on discretizations of continuous probability distributions while avoiding exponential growth in the representation size. Our approach leverages a low-rank tensor train decomposition to represent latent random variables compactly using Dirac deltas, enabling deterministic addition, subtraction and multiplication operations directly in the compressed format. Additionally, we present an efficient implementation strategy based on sparse matrices and specialized data structures that further enhances computational performance. We have designed the representation and distribution arithmetic methods to enable implementation within a general-purpose microprocessor or specialized hardware accelerator. 

This article presents theoretical analyses of the memory and computational complexity for common operations that show polynomial scaling under rank assumptions and illustrates how several statistics of the latent variable can be computed in polynomial memory and complexity. Numerical experiments on problems from randomized numerical linear algebra to stochastic differential equations with financial applications demonstrate orders-of-magnitude improvements in memory usage and computational time compared to conventional approaches, enabling tractable deterministic computations on discretized random variables in previously intractable dimensions.
\end{abstract}



\section{Introduction}\label{sec:introduction}
Many stochastic systems can be viewed as computations which propagate uncertainty through arithmetic operations on probability distributions \cite{zhang2023fundamentals, taylor1998introduction, grimmett2001probability, buzacott1993stochastic, glasserman2003monte, gardiner2009handbook}. The traditional approach to propagating probability distributions through computations is to use some variant of Monte-Carlo methods \cite{glasserman2003monte}. The \textit{sample-then-compute} approach of Monte Carlo methods requires enumerating the combinatorial space of joint probability distributions. This combinatorial challenge appears across diverse applications: financial portfolio modeling, where asset correlations interact through stochastic differential equations \cite{glasserman2003monte, oosterlee2019mathematical}; climate modeling, where uncertain parameters combine through complex mathematical relationships \cite{palmer2014climate, schneider2017uncertainty}; engineering reliability analysis, where material properties affect system performance through coupled physical processes \cite{ditlevsen2007structural, melchers2018structural}. In each case, the fundamental mathematical challenge remains the same: the enumeration of possible outcomes grows exponentially when combining multiple uncertain inputs.

\subsection{An Illustrative Example}
\label{sec:example}
The impact of such combinations becomes clear through a simple example. Rolling a single die (die A) yields 6 possible outcomes. A second die (die B) also yields 6 possible outcomes. Determining the possible outcomes of the sum of the two dice requires considering every combination of the outcomes of die A and die B, resulting in 36 ($6^2$) possible outcome combinations. Introducing a third die (die C) increases the number of outcome combinations to 216 ($6^3$). Although this remains tractable for small numbers of variables, the enumeration grows rapidly: for 10 dice, the total reaches over 60 million ($6^{10}$) outcome combinations, and for 1,000 dice, the number of combinations exceeds $6^{1000}$, far beyond the capacity of modern computing systems.

This illustrates a broader problem: combining multiple probability distributions by explicitly enumerating them (sampling) leads to exponential growth in the enumeration of possible outcome combinations, even though what one cares about might not specifically be to enumerate the outcomes, but rather, to obtain a probability density function (PDF) over the possible outcomes. As the number of variables increases, the memory and computational resources required to enumerate and process all possible outcome combinations become intractable.

\subsection{The Status Quo in Uncertainty Quantification}
Traditional uncertainty quantification methods address the challenge of the combinatorial explosion of the enumerated joint distribution support through various approaches, each with distinct limitations. Monte Carlo methods \cite{hammersley1964monte, liu2001monte} sample from input distributions and propagate uncertainty through repeated evaluations of some computational kernel. While statistically sound, these approaches require thousands to millions of evaluations to achieve reliable estimates, making them computationally expensive and inherently stochastic (repeating a million iteration Monte Carlo will never result in the same output distribution). The NIST Uncertainty Machine \cite{nist-mc} employs either Monte Carlo sampling or first-order Taylor approximations based on distribution moments, offering computational efficiency but suffering accuracy limitations with nonlinear transformations or non-Gaussian distributions. Polynomial chaos expansions \cite{xiu2002wiener} represent uncertain quantities using orthogonal polynomial bases, while interval arithmetic \cite{moore2009introduction} bounds uncertainty using intervals rather than full distributions. PaCAL \cite{korzen2014pacal} implements arithmetic operations on independent random variables by approximating probability density functions using Chebyshev polynomials \cite{trefethen2015computing, driscoll2014chebfun, wright2015extension}, enabling operations on continuous distributions but limiting applicability to parametric distributions. Additionally, PaCAL restricts its scope to independent random variables, which avoids the computational complexities of joint distributions but prevents its use in problems where variable dependencies are significant. Recent research has made similar and complementary observations.

It is notable that all existing uncertainty quantification methods are software-based implementations running on conventional processor architectures, inheriting limitations from the mismatch between continuous uncertainty propagation and discrete computational paradigms. This observation naturally motivates the exploration of processor-level approaches that can natively handle probabilistic computations.

\subsection{Discrete Representations of Probability Distributions}

Let $\{x_i\}$ denote support points and $\{m_i\}$ denote corresponding probability masses. Tsoutsouras \et~\cite{tsoutsouras2021laplace} introduced the Telescoping Torques Representation (TTR) to represent probability distributions using weighted sums of Dirac delta functions:

\begin{equation}
p_{X}(x) = \sum_{i=1}^{N} m_i \delta(x - x_i).
\end{equation}
Arithmetic operations such as addition, multiplication, and other transformations can be performed deterministically on these representations, avoiding the stochasticity of sampling-based methods. The accuracy of this approximation scales as $\mathcal{O}(1/N)$ \cite{bilgin_quantization_2025} where $N$ represents the number of representation points, allowing controllable precision through increased representation size. Tsoutsouras \et~\cite{tsoutsouras2021laplace} present engineering heuristics that allow this representation to handle a limited number of scenarios involving autocorrelation but not general correlation. Subsequent work by Bilgin \et~\cite{bilgin_patent} extended the basic one-dimensional representation (TTR) to joint distributions (jTTR), but faces a scalability limitation that prevents their application to high-dimensional joint distributions. 

For a $d$-dimensional distribution with $n$ support points per dimension, the total representation size is $N = n^d$. Computing the convolution of two such Dirac mixtures, each with $N$ support points, requires $\mathcal{O}(N^2)$ operations and produces a result with up to $N^2$ support points. Since $N$ must scale exponentially with dimension to maintain accuracy (typically $N = n^d$ for $n$ points per dimension), this becomes computationally intractable for high-dimensional distributions---precisely the regime where the example in Section \ref{sec:example} shows the combinatorial explosion becomes overwhelming.

\subsection{Tensor Train Methods for High-Dimensional Distributions}

Tensor decomposition techniques \cite{kolda2009tensor, tucker1966some, delathauwer2000multilinear, harshman1970foundations}, particularly the tensor train (TT) format \cite{oseledets2011tensor}, have emerged as powerful tools for representing high-dimensional tensors with exponential compression. The TT format represents $d$-dimensional tensors as products of three-dimensional tensors, achieving storage reduction from $\mathcal{O}(n^d)$ to $\mathcal{O}(dnr^2)$ where $r$ denotes the maximum TT-rank.

Recent work has explored tensor train applications in uncertainty quantification. One approach \cite{novikov2021tensor} demonstrated tensor train density estimation for approximating continuous probability distributions, showing that smooth distributions can be efficiently represented in TT format. Another method \cite{dolgov2020approximation} developed tensor train surrogates to enable more efficient sampling from high-dimensional continuous distributions, achieving significant improvements over traditional MCMC methods \cite{brooks2011handbook} by reducing autocorrelation times and enabling quasi-Monte Carlo integration \cite{dick2013high}.

However, these approaches remain fundamentally within the sampling paradigm. They require a known target probability density function to approximate and still rely on Monte Carlo estimation with its inherent stochasticity and convergence issues. While these methods successfully address the computational challenges of sampling from complex high-dimensional distributions, they do not eliminate the inherent limitations of stochastic approaches: non-reproducible results, convergence uncertainty, and the inability to perform exact arithmetic operations on distributions themselves.

\subsection{A New Finite-Dimensional Approach for High-Dimensional Distributions}
We address this scalability gap by developing the first tensor train framework for deterministic arithmetic operations on discretizations of high-dimensional continuous probability distributions. Our approach differs from both existing tensor train methods and previous deterministic approaches in several key ways.

Compared to existing tensor train probabilistic methods~\cite{novikov2021tensor, dolgov2020approximation}, our framework eliminates sampling entirely by performing deterministic arithmetic operations directly on discrete representations of continuous probability distributions (building on the prior work of single- and multi-dimensional Dirac mixtures \cite{tsoutsouras2021laplace, bilgin_patent}), while preserving the tensor train structure. Rather than building surrogates limited to known distributions, we adaptively construct tensor train representations through arithmetic operations themselves, enabling the combination of arbitrary distributions without requiring knowledge of their underlying probability density functions.

Compared to previous deterministic approaches like Tsoutsouras \et~\cite{tsoutsouras2021laplace}, which introduced arithmetic operations on discrete representations of continuous probability distributions using Dirac mixture representations (TTR), our method addresses the core scalability limitations that are still present in follow-on work (jTTR) by Bilgin \et~\cite{bilgin_patent}. While TTR \cite{tsoutsouras2021laplace} achieves excellent performance for low-dimensional problems, even its extension to high-dimensional distributions, jTTR~\cite{bilgin_patent}, cannot scale to high-dimensional joint distributions due to exponential growth in support points. Our tensor train TTR (TT-TTR) approach maintains the deterministic accuracy of Dirac mixture methods (TTR and jTTR), while achieving exponential compression for high-dimensional problems. Throughout this paper, we refer to the original Dirac mixture method as TTR (Telescoping Torques Representation) \cite{tsoutsouras2021laplace}, its joint distribution extension as jTTR \cite{bilgin_patent}, and our tensor train variant as TT-TTR.

The key innovation lies in combining the deterministic advantages of Dirac mixture representations with the exponential compression capabilities of tensor train decompositions. Consider a $d$-dimensional random variable $\mathbf{X} = (X_1, \ldots, X_d)$ with joint probability mass function $p_{\mathbf{X}}(\mathbf{x})$. While traditional Dirac mixture approaches require exponential storage $\mathcal{O}(n^d)$, our tensor train representation compresses the results of arithmetic operations on distributions, reducing storage to be linear in $d$\footnote{Although tensor train representations typically exhibit memory complexity $\mathcal{O}(dnr^2)$, our implementation strategy eliminates the dependence on $n$.}. This enables deterministic arithmetic with discretizations of probability distributions in dimensions previously considered intractable. Returning to the dice example of Section~\ref{sec:example}, we can now handle the $1{,}000$-dice scenario that exceeds $6^{1000}$ combinations. Previously this example required $6^{1000}$ memory elements, but using the tensor train approach we illustrate how this can be limited to $12{,}012$ memory elements through compressed tensor representations rather than explicit enumeration.

\subsection{Contributions}

The main contributions of this work are:

\paragraph{A review of the single- and multi-dimensional Dirac mixture representations} We review the TTR and jTTR Dirac mixture representations of Tsoutsouras \et~\cite{tsoutsouras2021laplace} and Bilgin \et~\cite{bilgin_patent} respectively, formulated in terms of matrices and tensors (Section~\ref{sec:operations}).

\paragraph{Novel tensor train arithmetic framework} We introduce the first tensor train TTR (TT-TTR) method for performing exact arithmetic operations (addition, subtraction and multiplication) directly on discrete representations of probability distributions, that  avoid the curse of dimensionality (Section~\ref{sec:tensor_train}).

\paragraph{Theoretical analysis and statistical computation} We provide a complexity analysis demonstrating polynomial scaling, and show how various statistics of the latent variables in TT-TTR format can be computed directly with polynomial computation and memory complexity (Section~\ref{sec:tensor_train} and Section~\ref{sec:complexity}).

\paragraph{Efficient implementation strategy and validation} We provide an efficient implementation strategy leveraging sparse matrices and specialized data structures, achieving further memory complexity reduction beyond the tensor train compression by a factor proportional to the discrete representation size. We perform comprehensive experiments across diverse applications from randomized linear algebra to financial stochastic modeling, demonstrating orders-of-magnitude improvements in memory usage and computational time, compared to the representations of Tsoutsouras \et~\cite{tsoutsouras2021laplace} and Bilgin \et~\cite{bilgin_patent} (Section~\ref{sec:algorithm}, Section~\ref{sec:complexity} and Section~\ref{sec:experiments}).

\paragraph{Deterministic arithmetic on probability distributions} Unlike stochastic Monte Carlo approaches that rely on sampling, our method performs exact reproducible arithmetic operations on discrete representations of probability distributions. While existing deterministic methods for distribution arithmetic suffer from the curse of dimensionality, our approach makes such computations tractable without sacrificing accuracy or introducing stochastic variability.

\subsection{Notation and Mathematical Preliminaries}

We use uppercase letters $X, Y, Z$ for random variables with probability mass functions $p_X(x), p_Y(y)$. Individual support points are denoted $x_i, x_j$ with masses $m_i, m_j$. We use $N_X$ for the number of support points of $X$, $n_k$ for support points in dimension $k$, and $d$ for the total number of independent variables.

Bold lowercase letters denote vectors: $\mathbf{s}^X \in \mathbb{R}^{N_X}$ and $\mathbf{m}^X \in \mathbb{R}^{N_X}$ are support and mass vectors for $X$. For multidimensional cases, $\mathbf{S}^X \in \mathbb{R}^{n_1 \times n_2}$ and $\mathcal{S}^X \in \mathbb{R}^{n_1 \times \cdots \times n_d}$ denote support matrices and tensors. Standard notation includes $\mathbf{1}$ for the $n$-dimensional ones vector, $\mathbf{A}^{\top}$ for transpose, and $\text{vec}(\cdot)$ for vectorization.

Calligraphic letters $\mathcal{M}, \mathcal{S}, \mathcal{I} \in \mathbb{R}^{n_1 \times \cdots \times n_d}$ denote tensors. The Kronecker product $\mathbf{A} \otimes \mathbf{B}$ for $\mathbf{A} \in \mathbb{R}^{n\times m}$ and $\mathbf{B} \in \mathbb{R}^{p\times q}$ yields a $pn\times qm$ matrix with blocks $a_{ij}\mathbf{B}$.

For tensor trains, $\mathbf{G}^{(k)}_{i_k} \in \mathbb{R}^{r_{k-1} \times r_k}$ denotes the $i_k$-th slice of core $\mathcal{G}^{(k)}$. We distinguish mass and support cores as $\mathcal{M}^{(k)}$ and $\mathcal{S}^{(k)}$. The TT-ranks $r_k$ connect modes $k$ and $k+1$ with $r_0 = r_d = 1$, and $r = \max_k r_k$.

\begin{definition}[Discrete Convolution]\label{def:convolution}
For probability mass functions $p_{X}$ and $p_{Y}$, their convolution $(p_{X} * p_{Y})(z)$ gives the mass function of $Z = X + Y$, computed as $\sum_x p_{X}(x) p_{Y}(z-x)$.
\end{definition}

\begin{definition}[Product Distribution]\label{def:product}
For independent variables $X$ and $Y$, the mass function of $W = XY$ is $p_{W}(w) = \sum_{x \neq 0} p_{X}(x) p_{Y}(w/x)$ for $w \neq 0$.
\end{definition}

The remainder of this paper is organized as follows. Section~\ref{sec:operations} presents the mathematical foundations for arithmetic operations on discretized random variables. Section~\ref{sec:tensor_train} develops the tensor train framework for efficient representating probability distributions. Section~\ref{sec:algorithm} presents implementation strategies that exploit sparsity and shared structure. Section~\ref{sec:complexity} provides a comparative complexity analysis, while Section~\ref{sec:experiments} demonstrates the effectiveness through numerical experiments. Section~\ref{sec:conclusions} concludes with implications and future directions.

\section{Basic Operations on Random Variables}\label{sec:operations}

We establish the mathematical foundation for arithmetic operations on discretized random variables based on the Dirac delta representation introduced by Tsoutsouras \et~\cite{tsoutsouras2021laplace} and Bilgin \et~\cite{bilgin_patent}. This foundation provides the basis for our tensor train computational framework.

\subsection{Discrete Dirac Delta Representation of Random Variables}

The Dirac mixture representations of Tsoutsouras \et~\cite{tsoutsouras2021laplace} and Bilgin \et~\cite{bilgin_patent} represent a random variable $X$ with finite support using a weighted sum of Dirac delta functions. Let $\{x_i\}_{i=1}^{N_{X}}$ denote the support points of $X$ and $\{m_i\}_{i=1}^{N_{X}}$ the corresponding probability masses satisfying $m_i \geq 0$ and $\sum_{i=1}^{N_{X}} m_i = 1$. Then the probability mass function can be expressed as
\begin{equation}\label{eq:dirac_representation}
p_{X}(x) = \sum_{i=1}^{N_{X}} m_i \delta(x - x_i),
\end{equation}
where $\delta(\cdot)$ is the Dirac delta function. Operations on the pdf that need to evaluate the density at a point between two Dirac deltas interpolate the cumulative density function (CDF) between neighboring Dirac deltas in the representation. The combination of the compact representation of the weighted Dirac mixture and the interpolation via the CDF during operations that need to integrate the pdf over its support, are the key insights of the Dirac mixture representation.

This representation provides controllable accuracy with approximation error scaling as $\mathcal{O}\left(\frac{1}{N_{X}}\right)$ where $N_{X}$ is the number of support points \cite{bilgin_quantization_2025}. The sparsity of the representation ensures that only non-zero probability regions require storage, and operations can be performed deterministically on support points without sampling.

Given the Dirac delta representation \eqref{eq:dirac_representation}, the key challenge lies in determining the optimal placement of support points $\{x_i\}$ and their associated masses $\{m_i\}$. While alternatives such as regularly-quantized histograms and adaptive binning exist, we employ the Telescoping Torques Representation (TTR) \cite{tsoutsouras2021laplace} as the basis for the tensor train representation due to its superior performance under arithmetic operations \cite{bilgin_quantization_2025}.

\subsubsection{The Telescoping Torques Representation (TTR) Construction}
The approaches of Tsoutsouras \et~\cite{tsoutsouras2021laplace} and Bilgin \et~\cite{bilgin_patent} construct Dirac mixture representations using the Telescoping Torques Representation (TTR) algorithm. TTR is a domain-splitting, divide-and-conquer algorithm that recursively partitions probability distributions based on their mean values.

For empirical samples $\{s_1, s_2, \ldots, s_M\}$ with equal weights $\frac{1}{M}$, TTR applies the recursive mean-based splitting strategy. This process continues until reaching the desired discretization level $n$, producing $2^n$ Dirac measures that capture both the central tendencies and probability masses of the original distribution. For parametric distributions $P$, TTR leverages analytic methods when possible to determine the set of Dirac delta positions and masses that optimally represent the said distribution at the given number of Dirac deltas (i.e., the representation size). More precisely, TTR uses closed-form expressions for the cumulative distribution function and first moment function to compute exact split points and masses, employing location-scale transformations for distributions in these families. 
We refer the reader to Tsoutsouras \et~\cite{tsoutsouras2021laplace} for a detailed description of the TTR algorithm.

\subsection{Arithmetic Operations in Vector Form}
\label{sec:arithmetic}
The Dirac delta representation naturally discretizes the continuous convolution and product operations. For random variables $X$ and $Y$ with representations $p_X(x) = \sum_{i=1}^{N_X} m_i^X \delta(x - x_i)$ and $p_Y(y) = \sum_{j=1}^{N_Y} m_j^Y \delta(y - y_j)$, arithmetic operations reduce to pairwise combinations \cite{tsoutsouras2021laplace}. The convolution $(p_X * p_Y)(z) = \sum_{i,j} m_i^X m_j^Y \delta(z - (x_i + y_j))$ for addition and the product distribution $p_W(w) = \sum_{i,j} m_i^X m_j^Y \delta(w - x_i y_j)$ for multiplication reveal a common structure: support points combine via the arithmetic operation while probability masses multiply.

This pairwise structure motivates a vector representation. For a random variable $X$, we define
\begin{equation}\label{eq:vector_representation}
\mathbf{s}^{X} = [x_1, x_2, \ldots, x_{N_{X}}]^{\top}, \quad \mathbf{m}^{X} = [m_1, m_2, \ldots, m_{N_{X}}]^{\top},
\end{equation}
where $\|\mathbf{m}^{X}\|_1 = 1$ and $m_i \geq 0$. The pairwise operations then naturally express as outer products. For the sum $Z = X + Y$, the support matrix
\begin{equation}\label{eq:support_matrix_add}
\mathbf{S}^{Z} = \mathbf{s}^{X} \mathbf{1}_{N_{Y}}^{\top} + \mathbf{1}_{N_{X}} (\mathbf{s}^{Y})^{\top}
\end{equation}
captures all pairwise sums $x_i + y_j$. Crucially, $\mathbf{S}^{Z}$ decomposes as the sum of two rank-1 matrices, yielding rank at most 2. For multiplication $W = XY$, the support matrix
\begin{equation}\label{eq:support_matrix_mult}
\mathbf{S}^{W} = \mathbf{s}^{X} (\mathbf{s}^{Y})^{\top}
\end{equation}
forms a rank-1 outer product. The probability mass matrix maintains rank-1 structure for both operations:
\begin{equation}\label{eq:mass_matrix}
\mathbf{M}^{Z} = \mathbf{M}^{W} = \mathbf{m}^{X} (\mathbf{m}^{Y})^{\top}.
\end{equation}
Vectorizing these matrices yields the final representations $\mathbf{s}^{\ast} = \text{vec}(\mathbf{S}^{\ast})$ and $\mathbf{m}^{\ast} = \text{vec}(\mathbf{M}^{\ast})$ with $N_X N_Y$ entries. Subtraction follows identically with $\mathbf{S}^{Z} = \mathbf{s}^{X} \mathbf{1}_{N_{Y}}^{\top} - \mathbf{1}_{N_{X}} (\mathbf{s}^{Y})^{\top}$.

\begin{figure}[htbp]
\centering
\begin{subfigure}{0.48\textwidth}
    \includegraphics[width=\textwidth]{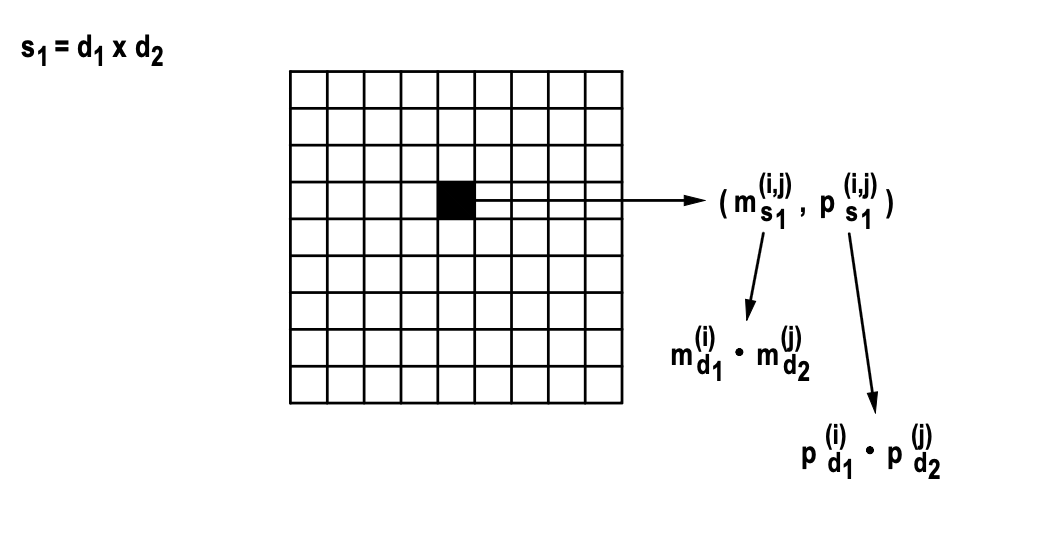}
\end{subfigure}
\hfill
\begin{subfigure}{0.48\textwidth}
    \includegraphics[width=\textwidth]{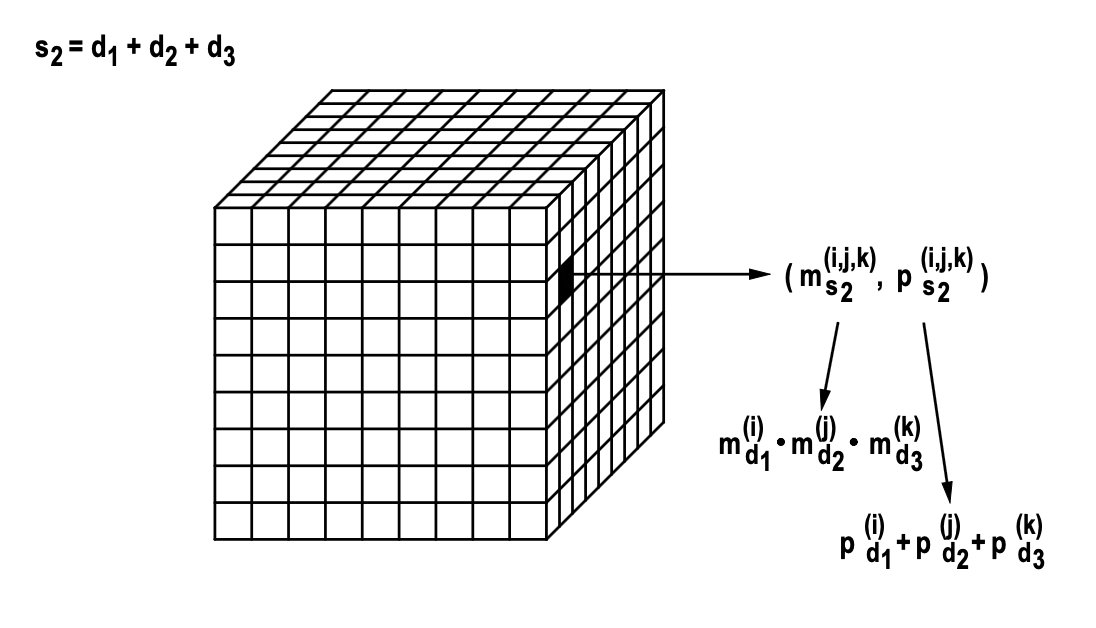}
\end{subfigure}
\caption{Matrix and tensor structures resulting from (left) multiplication of two and (right) addition of three independent TTR-discretized distributions.}
\label{fig:randomMatrix}
\end{figure}

The extension to $d$ independent random variables $X_1, X_2, \ldots, X_d$, known as jTTR \cite{bilgin_patent}, produces $d$-dimensional tensors while preserving the low-rank structure. The joint probability mass tensor
\begin{equation}\label{eq:joint_mass_tensor}
\mathcal{M} = \mathbf{m}^{X_1} \otimes \mathbf{m}^{X_2} \otimes \cdots \otimes \mathbf{m}^{X_d} \in \mathbb{R}^{n_1 \times n_2 \times \cdots \times n_d}
\end{equation}
maintains rank 1 as a Kronecker product of vectors. For the sum $Z = \sum_{k=1}^d X_k$, the support tensor
\begin{equation}\label{eq:addition_tensor}
\mathcal{S}^{Z} = \sum_{k=1}^d \left( \bigotimes_{j=1}^{k-1} \mathbf{1}_{n_j} \right) \otimes \mathbf{s}^{X_k} \otimes \left( \bigotimes_{j=k+1}^{d} \mathbf{1}_{n_j} \right)
\end{equation}
generalizes the rank-2 matrix structure, while the product $W = \prod_{k=1}^d X_k$ yields the rank-1 tensor
\begin{equation}\label{eq:multiplication_tensor}
\mathcal{S}^{W} = \mathbf{s}^{X_1} \otimes \mathbf{s}^{X_2} \otimes \cdots \otimes \mathbf{s}^{X_d}.
\end{equation}
These operations produce $\prod_{k=1}^d n_k$ support points and masses, manifesting the curse of dimensionality. With $n_k = 10$ points per variable and $d = 10$ variables, storing $10^{10}$ values exceeds computational limits. However, the preserved low-rank structure (rank 2 for addition supports, rank 1 for multiplication supports and all mass tensors) provides the foundation for the tensor train TTR representation developed in Section~\ref{sec:tensor_train}.

\section{Tensor Train Framework for Arithmetic on Discretized Probability Distributions}\label{sec:tensor_train}

We develop a tensor train framework for performing arithmetic operations on discrete representations of probability distributions while avoiding the exponential scaling that plagues traditional approaches. Our method constructs tensor train representations that encode arithmetic operations directly in compressed form and adaptively preserves this structure throughout computation. The resulting TT-TTR representation remains fundamentally a Dirac mixture within the TTR family introduced by Tsoutsouras et al.~\cite{tsoutsouras2021laplace} (TTR) and extended by Bilgin et al.~\cite{bilgin_patent} (jTTR), preserving the same approximation guarantees and mean-preserving properties. The key innovation lies in the tensor train factorization that provides a memory-efficient encoding of the latent joint distribution, reducing storage from exponential to polynomial complexity while maintaining the mathematical equivalence to the full Dirac mixture representation. Specifically, we do not follow an \textit{expand-then-compress} methodology. Instead, we update the tensor train cores directly in the compressed space to perform arithmetic operations, without ever forming the full tensor. This is crucial, as all memory benefits gained from the tensor train representation would be lost if we ever formed the full tensor, even if the final result were stored in compressed form.

\subsection{Tensor Train Decomposition}
A $d$-dimensional tensor $\mathcal{T} \in \mathbb{R}^{n_1 \times n_2 \times \cdots \times n_d}$ has a tensor train decomposition \cite{oseledets2011tensor} if it can be written as\footnote{We explicitly use an equals sign here, however practitioners often apply further truncations of the cores, which would turn this expression into an approximation.}:
\begin{equation}\label{eq:tt_decomposition}
\mathcal{T}_{i_1, i_2, \ldots, i_d} = \mathbf{G}^{(1)}_{i_1} \mathbf{G}^{(2)}_{i_2} \cdots \mathbf{G}^{(d)}_{i_d},
\end{equation}
where $\mathbf{G}^{(k)}_{i_k} \in \mathbb{R}^{r_{k-1} \times r_k}$ are matrix slices of the 3-way tensor cores $\mathcal{G}^{(k)} \in \mathbb{R}^{r_{k-1} \times n_k \times r_k}$, with boundary conditions $r_0 = r_d = 1$. The integers $r_k$ for $k = 1, \ldots, d-1$ are the TT-ranks, and $r = \max_k r_k$ is the maximum rank. Storage complexity reduces from $\mathcal{O}(\prod_{k=1}^d n_k)$ to $\mathcal{O}(dnr^2)$, achieving exponential compression when $r \ll n$.

\subsection{Tensor Train Representation of Probability Distributions}
We represent the results of arithmetic operations on discrete representations of probability distributions using tensor train decompositions. Following the tensor formulation from Section~\ref{sec:arithmetic}, when combining $d$ independent random variables $X_1, \ldots, X_d$ with support vectors $\mathbf{s}^{X_k} \in \mathbb{R}^{n_k}$ and mass vectors $\mathbf{m}^{X_k} \in \mathbb{R}^{n_k}$, we construct two tensors to represent the outcome distribution. The support tensor $\mathcal{S} \in \mathbb{R}^{n_1 \times \cdots \times n_d}$ encodes the possible outcome values, where each entry $\mathcal{S}_{i_1, \ldots, i_d}$ represents the result of applying the arithmetic operation to the support points $s_{i_1}^{X_1}, \ldots, s_{i_d}^{X_d}$. The probability mass tensor $\mathcal{M} \in \mathbb{R}^{n_1 \times \cdots \times n_d}$ maintains the joint probability structure, where $\mathcal{M}_{i_1, \ldots, i_d} = \prod_{k=1}^d m_{i_k}^{X_k}$ due to independence. Let $\mathbf{S}^{(k)}_{i_k}$ and $\mathbf{M}^{(k)}_{i_k}$ be slices of the tensor train cores $\mathcal{S}^{(k)}$ and $\mathcal{M}^{(k)}$ respectively, then both tensors admit tensor train decompositions:
\begin{equation*}
\mathcal{S}_{i_1, \ldots, i_d} = \mathbf{S}^{(1)}_{i_1} \mathbf{S}^{(2)}_{i_2} \cdots \mathbf{S}^{(d)}_{i_d}, \quad
\mathcal{M}_{i_1, \ldots, i_d} = \mathbf{M}^{(1)}_{i_1} \mathbf{M}^{(2)}_{i_2} \cdots \mathbf{M}^{(d)}_{i_d}.
\end{equation*}
As we demonstrate in the following sections, the structure of arithmetic operations leads to low-rank decompositions, with the probability mass tensor always achieving rank 1 due to independence, while the support tensor rank depends on the specific network of operations.

\subsection{Pure Arithmetic Operations}

We first analyze pure operations involving multiple independent variables, which exhibit particularly efficient tensor train structures.

\begin{theorem}[Addition in Tensor Train Format]\label{thm:tt_addition}
Let $X_1, \ldots, X_d$ be independent random variables. The sum $Z = \sum_{k=1}^d X_k$ admits tensor train representations with support tensor rank 2 and probability mass tensor rank 1:
\begin{equation}\label{eq:addition_tt}
\mathcal{S}^{Z} = \mathcal{S}^{(1)} \circ \cdots \circ \mathcal{S}^{(d)}, \quad \mathcal{M}^{Z} = \mathcal{M}^{(1)} \circ \cdots \circ \mathcal{M}^{(d)}
\end{equation}
where the support cores $\mathcal{S}^{(k)}$ have slices:
\footnotesize
\begin{equation*}
\mathbf{S}^{(1)}_{i_1} = [\mathbf{s}^{X_1}_{i_1}, 1] \in \mathbb{R}^{1 \times 2}, \quad
\mathbf{S}^{(k)}_{i_k} = \begin{bmatrix} 1 & 0 \\ \mathbf{s}^{X_k}_{i_k} & 1 \end{bmatrix} \in \mathbb{R}^{2 \times 2} \text{ for } k = 2, \ldots, d-1, \quad
\mathbf{S}^{(d)}_{i_d} = \begin{bmatrix} 1 \\ \mathbf{s}^{X_d}_{i_d} \end{bmatrix} \in \mathbb{R}^{2 \times 1}
\end{equation*}
\normalsize
and the mass cores $\mathcal{M}^{(k)}$ have slices:
\begin{align}
\mathbf{M}^{(k)}_{i_k} &= \mathbf{m}^{X_k}_{i_k} \in \mathbb{R}^{1 \times 1}, \quad k = 1, \ldots, d
\end{align}
\end{theorem}

\begin{proof}
The support tensor construction follows from the tensor train addition algorithm 
\cite{oseledets2011tensor}, where we encode the accumulation $\sum_{k=1}^d \mathbf{s}^{X_k}_{i_k}$ 
using augmented cores with rank 2 that implement a running sum mechanism. The probability mass
tensor maintains rank 1 since $\mathcal{M}^{Z}_{i_1,\ldots,i_d} = \prod_{k=1}^d \mathbf{m}^{X_k}_{i_k}$ 
forms a Kronecker product structure.
\end{proof}

\begin{figure}[htbp]
\centering
\begin{subfigure}{0.48\textwidth}
    \centering
    \includegraphics[width=\textwidth]{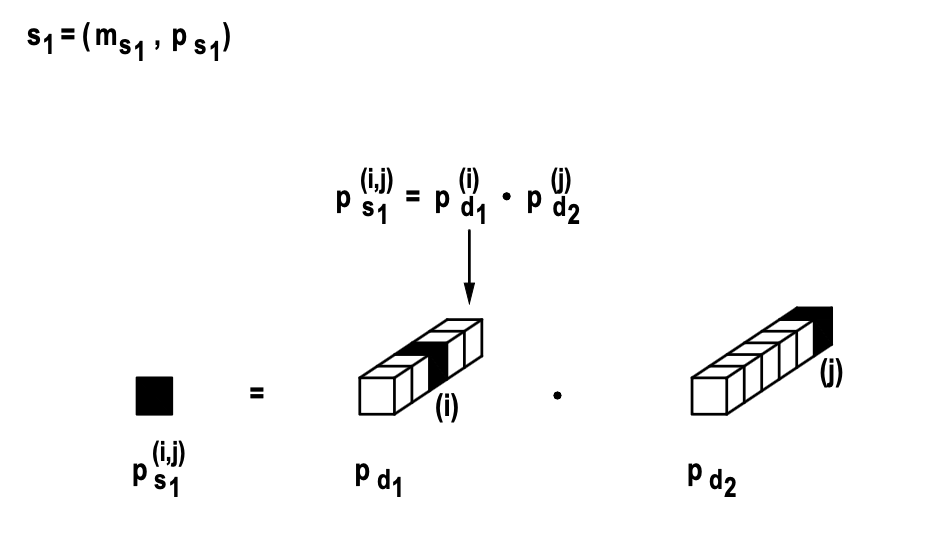}
\end{subfigure}
\hfill
\begin{subfigure}{0.48\textwidth}
    \centering
    \includegraphics[width=\textwidth]{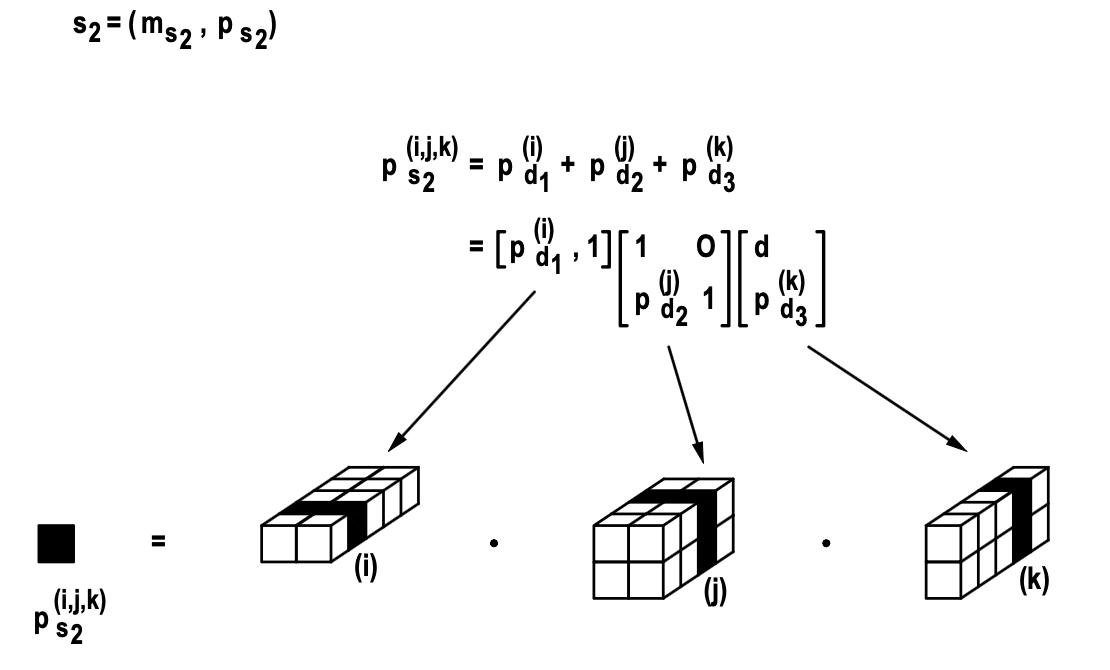}
\end{subfigure}
\caption{Tensor train decompositions: (left) applying Theorem~\ref{thm:tt_multiplication} to the multiplication in Figure~\ref{fig:randomMatrix}, left; (right) applying Theorem~\ref{thm:tt_addition} to the addition in Figure~\ref{fig:randomMatrix}, right.}
\label{fig:randomTT}
\end{figure}

\begin{theorem}[Multiplication in Tensor Train Format]\label{thm:tt_multiplication}
For the product $W = \prod_{k=1}^d X_k$, both support and probability mass tensors achieve rank 1 with cores having slices:
\begin{equation}
\mathbf{S}^{(k)}_{i_k} = \mathbf{s}^{X_k}_{i_k}, \quad \mathbf{M}^{(k)}_{i_k} = \mathbf{m}^{X_k}_{i_k} \in \mathbb{R}^{1 \times 1}, \quad k = 1, \ldots, d
\end{equation}
\end{theorem}

\begin{proof}
Both tensors follow identical Kronecker product structures. Each entry becomes $\mathcal{S}^{W}_{i_1,\ldots,i_d} = \prod_{k=1}^d \mathbf{s}^{X_k}_{i_k}$ and $\mathcal{M}^{W}_{i_1,\ldots,i_d} = \prod_{k=1}^d \mathbf{m}^{X_k}_{i_k}$, both achievable through rank-1 tensor train contractions.
\end{proof}

We now generalize to arbitrary arithmetic expressions combining both operations.

\subsection{General Arithmetic Operations}

Pure addition or multiplication operations maintain low tensor train ranks: rank 2 for addition and rank 1 for multiplication, regardless of the number of variables involved. This property enables efficient construction of tensor train decompositions for networks composed entirely of one operation type, as detailed in Section \ref{sec:algorithm}. For general arithmetic expressions mixing different operations, or when combining variables that themselves result from mixed operations, we require the following more general construction based on \cite{oseledets2011tensor}.

\begin{theorem}[General Tensor Train Arithmetic]\label{thm:tt_general}
Let $X$ and $Y$ be random variables whose distributions are represented by tensors $\mathcal{X}, \mathcal{Y} \in \mathbb{R}^{n_1 \times \cdots \times n_d}$ with tensor train decompositions:
\begin{equation}
\mathcal{X}_{i_1,\ldots,i_d} = \mathbf{X}^{(1)}_{i_1} \cdots \mathbf{X}^{(d)}_{i_d}, \quad \mathcal{Y}_{i_1,\ldots,i_d} = \mathbf{Y}^{(1)}_{i_1} \cdots \mathbf{Y}^{(d)}_{i_d},
\end{equation}
where both tensors have dimensions $n_1 \times \cdots \times n_d$ with maximum ranks $r_{X}$ and $r_{Y}$ respectively. Arithmetic operations preserve the tensor train structure; addition ($Z = X + Y$) merges cores with rank $r_{Z} \leq r_{X} + r_{Y}$:
\footnotesize
\begin{equation}\label{eq:addition_cores}
\mathbf{Z}^{(k)}_{i_k} = \begin{cases}
\begin{bmatrix} \mathbf{X}^{(1)}_{i_1} & \mathbf{Y}^{(1)}_{i_1} \end{bmatrix} & k = 1 \\[1em]
\begin{bmatrix} \mathbf{X}^{(k)}_{i_k} & \mathbf{0} \\ \mathbf{0} & \mathbf{Y}^{(k)}_{i_k} \end{bmatrix} & k = 2, \ldots, d-1 \\[1em]
\begin{bmatrix} \mathbf{X}^{(d)}_{i_d} \\ \mathbf{Y}^{(d)}_{i_d} \end{bmatrix} & k = d,
\end{cases}
\end{equation}
\normalsize
whereas multiplication ($W = X \cdot Y$) combines cores via Kronecker products with rank $r_{W} \leq r_{X} \cdot r_{Y}$:
\begin{equation}\label{eq:multiplication_cores}
\mathbf{W}^{(k)}_{i_k} = \mathbf{X}^{(k)}_{i_k} \otimes \mathbf{Y}^{(k)}_{i_k}, \quad k = 1, \ldots, d.
\end{equation}
\end{theorem}

\begin{proof}
The results follow directly from \cite[Section 4]{oseledets2011tensor}.
\end{proof}
The tensor train arithmetic operations above require that both operands have identical tensor dimensions $\left(n_1, n_2, \cdots, n_d\right)$. This constraint arises because element-wise operations demand matching indices across all modes. However, in practice, random variables often depend on different subsets of underlying factors or arise from distinct computational paths. The following section presents our correlation tracking framework, which identifies shared dependencies between variables and enables operations on random variables with different dimensional structures by aligning their tensor representations.

\subsection{Keeping Track of Correlations Between Variables}

Real-world computations involve variables with different dependencies. Consider $Z = X + Y$ where $X$ depends on $(T, H, P)$ (temperature, humidity, pressure) and $Y$ depends on $(T, W, S)$ (temperature, wind, solar radiation). Their tensor representations occupy different spaces: $\mathcal{X} \in \mathbb{R}^{n_T \times n_H \times n_P}$ and $\mathcal{Y} \in \mathbb{R}^{n_T \times n_W \times n_S}$. The result $Z$ depends on all five factors, requiring representation in $\mathbb{R}^{n_T \times n_H \times n_P \times n_W \times n_S}$.

The tensor train decomposition exposes this dependency structure through cores $\mathbf{X}^{(k)}$ corresponding to specific ancestral variables. We handle shared factors by embedding both operands in the full space and inserting identity cores for absent dimensions.

For $X = T + H + P$ and $Y = T \times W \times S$, following Theorems~\ref{thm:tt_addition} and \ref{thm:tt_multiplication}, the support tensor $\mathcal{X}$ has rank-2 cores while $\mathcal{Y}$ has rank-1 cores. Let $\mathcal{I}^{(k)}$ and $\mathcal{J}^{(k)}$ denote appropriately sized identity cores (with slices $\mathbf{I}^{(k)}_{i}$ and $\mathbf{J}^{(k)}_{j}$ respectively), then to compute $Z = X \star Y$, we embed both tensors in the unified space $(T, H, P, W, S)$:
\begin{equation}
\widetilde{\mathcal{X}}_{i_T, i_H, i_P, i_W, i_S} = \mathbf{X}^{(T)}_{i_T} \mathbf{X}^{(H)}_{i_H} \mathbf{X}^{(P)}_{i_P} \mathbf{I}^{(W)}_{i_W} \mathbf{I}^{(S)}_{i_S}
\end{equation}
\begin{equation}
\widetilde{\mathcal{Y}}_{i_T, i_H, i_P, i_W, i_S} = \mathbf{Y}^{(T)}_{i_T} \mathbf{J}^{(H)}_{i_H} \mathbf{J}^{(P)}_{i_P} \mathbf{Y}^{(W)}_{i_W} \mathbf{Y}^{(S)}_{i_S}.
\end{equation}

\begin{theorem}[Tensor Train Alignment with Correlation Preservation]\label{thm:tt_alignment}
Let $\mathcal{X} \in \mathbb{R}^{\prod_{k \in \mathcal{I}_X} n_k}$ and $\mathcal{Y} \in \mathbb{R}^{\prod_{k \in \mathcal{I}_Y} n_k}$ be tensors with tensor train decompositions over index sets $\mathcal{I}_X$ and $\mathcal{I}_Y$. Define $\mathcal{I} = \mathcal{I}_X \cup \mathcal{I}_Y$ with ordering $\prec$. 

The aligned tensor train representations operate in $\mathbb{R}^{\prod_{k \in \mathcal{I}} n_k}$ via:
\begin{equation}
\widetilde{\mathcal{X}}^{(k)} = \begin{cases}
\mathcal{X}^{(k)} & k \in \mathcal{I}_X \\
\mathcal{I}_{r_{\prec k} \times r_{\succ k}} & k \in \mathcal{I} \setminus \mathcal{I}_X
\end{cases}
\end{equation}
where $r_{\prec k}$ and $r_{\succ k}$ denote ranks from adjacent cores in $\mathcal{I}_X$. The aligned tensors preserve original values on respective subspaces and handle joint distributions over shared variables $k \in \mathcal{I}_X \cap \mathcal{I}_Y$ while maintaining independence of disjoint variables.
\end{theorem}

\begin{proof}
For index $(i_k)_{k \in \mathcal{I}}$ in the unified space, tensor train contraction yields:
\begin{equation}
\widetilde{\mathcal{X}}_{(i_k)_{k \in \mathcal{I}}} = \prod_{k \in \mathcal{I}_X} \mathbf{X}^{(k)}_{i_k} \prod_{k \in \mathcal{I} \setminus \mathcal{I}_X} \mathbf{I}^{(k)}_{i_k}
\end{equation}
Since identity cores satisfy $\mathbf{I}^{(k)}_{i_k} = \mathbf{I}$ and contraction preserves the product:
\begin{equation}
\widetilde{\mathcal{X}}_{(i_k)_{k \in \mathcal{I}}} = \prod_{k \in \mathcal{I}_X} \mathbf{X}^{(k)}_{i_k} = \mathcal{X}_{(i_k)_{k \in \mathcal{I}_X}}.
\end{equation}
\end{proof}
\begin{remark}
Consecutive identity cores $\mathbf{I}^{(k)}, \ldots, \mathbf{I}^{(k+m)}$ allow intermediate ranks of 1, requiring only boundary matching: $\mathbf{I}^{(k)}_{r_{\prec k} \times 1} \circ \mathbf{I}^{(k+1)}_{1 \times 1} \circ \cdots \circ \mathbf{I}^{(k+m)}_{1 \times r_{\succ(k+m)}}$.
\end{remark}

\subsection{Efficient Statistical Computations}

Computing statistical properties of discrete representations of probability distributions typically requires summing over exponentially many support points \cite{bilgin_patent}. For distributions with support tensor $\mathcal{S}$ and mass tensor $\mathcal{M}$ in tensor train format, we reduce these computations to multidimensional contractions. The tensor train structure enables operations in $\mathcal{O}(dnr^3)$ time rather than $\mathcal{O}(n^d)$, making high-dimensional statistical analysis computationally feasible \cite{oseledets2011tensor}.

\begin{theorem}[Efficient Moment Computation]\label{thm:moments}
Let $X$ denote a random variable with discrete representation having support tensor $\mathcal{S}$ and probability mass tensor $\mathcal{M}$ in tensor train format. The $k$-th moment approximation is computed as:
\begin{equation}\label{eq:moment_computation}
\mu_k^{\text{disc}} = \frac{\langle \mathcal{S}^{\odot k}, \mathcal{M} \rangle}{\langle \mathbf{1}, \mathcal{M} \rangle}
\end{equation}
where $\mathcal{S}^{\odot k}$ denotes the $k$-th Hadamard power and $\langle \cdot, \cdot \rangle$ the tensor train inner product, with complexity $\mathcal{O}(dnr^3)$.
\end{theorem}

\begin{proof}
The result follows directly from the definition of moments for random variables represented using Dirac deltas \cite{bilgin_patent, tsoutsouras2021laplace}, and the definition and complexity of tensor train inner products \cite{oseledets2011tensor}.
\end{proof}

Computing $\mathcal{S}^{\odot k}$ can increase ranks up to $r^k$, though actual growth depends on the structure of $\mathcal{S}$. Thus the effective rank in the $\mathcal{O}(dnr^3)$ expression may be significantly larger than the original. Section~\ref{sec:algorithm} addresses this through an efficient implementation strategy.

\begin{theorem}[Efficient Covariance Computation]\label{thm:covariance}
Let $X$ and $Y$ be random variables represented in tensor train format, with $Z = XY$ computed via Theorem~\ref{thm:tt_general}. If $\mathcal{S}_Z$ and $\mathcal{M}_Z$ represent the support and mass tensors of $Z$, then:
\begin{equation}
\text{Cov}^{\text{disc}}(X,Y) = \frac{\langle \mathcal{S}_Z, \mathcal{M}_Z \rangle}{\langle \mathbf{1}, \mathcal{M}_Z \rangle} - \mu_X^{\text{disc}} \mu_Y^{\text{disc}}
\end{equation}
where $\mu_X^{\text{disc}}$ and $\mu_Y^{\text{disc}}$ are computed via Theorem~\ref{thm:moments}.
\end{theorem}

\begin{proof}
By definition, $\text{Cov}(X,Y) = \mathbb{E}[XY] - \mathbb{E}[X]\mathbb{E}[Y]$. Theorem~\ref{thm:tt_general} computes $Z = XY$, automatically handling shared dependencies through alignment. The first term equals the first moment of $Z$ computed as $\langle \mathcal{S}_Z, \mathcal{M}_Z \rangle / \langle \mathbf{1}, \mathcal{M}_Z \rangle$. All operations maintain $\mathcal{O}(dnr^3)$ complexity through the tensor train structure.
\end{proof}

\section{An Efficient Implementation Strategy}\label{sec:algorithm}

The tensor train cores from Section~\ref{sec:tensor_train} contain substantial structure that we exploit for efficiency. In particular, the cores from Theorem~\ref{thm:tt_general} contain many zero and one entries due to block-stacking and Kronecker product operations, making explicit storage wasteful.

The implementation strategy we present directly supports hardware realization. Each optimization we develop corresponds to specific microarchitectural features that enable high-performance distribution arithmetic to be implemented in a hardware accelerator or as part of the microarchitecture of a future general-purpose microprocessor.

\subsection{Sparse Representation of Tensor Train Cores}
\label{subsec:sparsecore}

We introduce a sparse representation that tracks only essential information:

\begin{definition}[Sparse Core Representation]
\label{def:sparse_core}
Let $(i,\ell)$ denote a matrix position, $\alpha$ the power of the support vector, and $c$ the coefficient. For a tensor train core $\mathcal{G}^{(k)}$ corresponding to ancestral variable $k$, we store:
\begin{equation}
\mathcal{G}^{(k)}_{\text{sparse}} = \{(i,\ell, \alpha, c) : \mathbf{G}^{(k)}_{i_k}[i,\ell] = c \cdot (\mathbf{s}^{(k)}_{i_k})^\alpha \neq 0\}.
\end{equation}
\end{definition}
For example, consider a rank-2 core slice from pure addition (Theorem~\ref{thm:tt_addition}):
$$
\mathbf{G}^{(k)}_{i_k} = \begin{bmatrix} 1 & 0 \\ 7\mathbf{s}^{(k)}_{i_k} & 1 \end{bmatrix}.
$$
Our sparse representation stores only $\{(1,1,0,1), (2,1,1,7), (2,2,0,1)\}$, indicating that position $(1,1)$ contains $1 \cdot (\mathbf{s}^{(k)}_{i_k})^0 = 1$, position $(2,1)$ contains $7 \cdot (\mathbf{s}^{(k)}_{i_k})^1$, and position $(2,2)$ contains $1$. The zero entry is not stored.

This representation reduces memory from $\mathcal{O}(r_{k-1} \cdot n_k \cdot r_k)$ to $\mathcal{O}(\text{\textit{nnz}})$ per core, where \textit{nnz} denotes the number of structurally nonzero entries per slice (not the entire tensor). The arithmetic operations from Theorem~\ref{thm:tt_general} naturally map to this representation through index transformations and coefficient/power accumulations.

\subsection{Tracking Pure Arithmetic Operations}
\label{subsec:operations}

The sparse representation becomes particularly effective when tracking the arithmetic heritage of each tensor train. Variables from pure operations exhibit predictable low-rank cores that we exploit (Theorems~\ref{thm:tt_addition} and~\ref{thm:tt_multiplication}).

\begin{definition}[Operation Type Tracking]
\label{def:op_tracking}
For each tensor train representation, we maintain an operation type flag $\text{OpType} \in \{\texttt{ADD}, \texttt{MULT}, \texttt{MIXED}\}$ with corresponding sparse pattern guarantees: \texttt{ADD} operations yield core slices with at most 3 nonzero entries in a $2 \times 2$ pattern (Theorem~\ref{thm:tt_addition}), \texttt{MULT} operations yield exactly 1 nonzero entry per slice (Theorem~\ref{thm:tt_multiplication}), and \texttt{MIXED} operations have no structural guarantees and require general rank propagation (~\ref{thm:tt_general}).
\end{definition}

This tracking enables optimized arithmetic algorithms. When combining variables with the same pure arithmetic history, we create new cores for uncorrelated variables (rank-2 for addition, rank-1 for multiplication), while for correlated variables sharing common ancestors, we directly update support vector coefficients and powers in existing cores without constructing general block-diagonal or Kronecker product structures.

\subsection{Exploiting Shared Ancestral Distributions}
\label{subsec:registry}

Many applications involve multiple independent random variables drawn from identical distributions. For example, Geometric Brownian Motion introduces independent and identically distributed (i.i.d.) Gaussian variables at each time step, all sharing the same discretization.

\begin{definition}[Distribution Type Registry]
\label{def:dist_registry}
We maintain a registry of unique distribution types, where each type $j$ stores
\begin{equation}
\text{DistType}_j = \{\mathbf{s}_j, \mathbf{m}_j\}.
\end{equation}
Variables $X_1, \ldots, X_m$ sharing distribution type $j$ reference this shared representation rather than storing duplicate vectors.
\end{definition}

This optimization is particularly effective in stochastic differential equations, randomized linear algebra, and Monte Carlo integration where many variables share common distributional forms.

\subsection{Efficient Moment Computation via Cached Contractions}
\label{subsec:moments}

The sparse core representation and distribution registry synergize to enable efficient moment computations. When multiple variables share the same distribution type (e.g., repeated Gaussian increments in Brownian motion), they reference identical support and mass vectors. This enables computing expensive contractions like $\sum_i s_i^m m_i$ once per distribution type rather than per variable, reducing preprocessing cost by a factor of $d/k$ for $d$ variables with $k \le d$ unique distributions.

By Theorem~\ref{thm:moments}, computing the $m$-th moment requires evaluating $\langle \mathcal{S}^{\odot m}, \mathcal{M} \rangle$, which decomposes into three steps: (1) compute the $m$-th Hadamard power of $\mathcal{S}$ using Theorem~\ref{thm:tt_general}\footnote{For pure $\texttt{mult}$ operations, we apply Theorem~\ref{thm:tt_multiplication}.}, (2) contract each tensor core along its fiber dimension, and (3) multiply the resulting matrices to obtain a scalar. The fiber contraction for core $k$ computes:
\begin{equation}
\mathbf{G}^{(k)} = \sum_{i_k=1}^{n_k} \mathbf{G}^{(k)}_{i_k} \cdot m_{i_k}.
\end{equation}
One implementation strategy, which we employ in the evaluation of Section~\ref{sec:experiments}, is to accelerate this computation through precomputed contractions. Each sparse core entry $(i,\ell,\alpha,c)$ indicates that position $(i,\ell)$ contains $c \cdot (\mathbf{s}_{i_k})^{\alpha}$. For moment computation, we need $c \cdot (\mathbf{s}_{i_k})^{\alpha} \cdot \mathbf{m}_{i_k}$. Since cores sharing distribution type $j$ use identical vectors, we precompute:
\begin{equation}
\text{Cache}[j,\alpha] = \langle \mathbf{s}_j^{\alpha}, \mathbf{m}_j \rangle = \sum_{i=1}^{n_j} (\mathbf{s}_j[i])^{\alpha} \cdot \mathbf{m}_j[i]
\end{equation}
for each distribution type $j$ and power $\alpha$. This transforms fiber contraction into a lookup: each sparse entry $(i,\ell,\alpha,c)$ contributes $c \cdot \text{Cache}[j, \alpha]$ to position $(i,\ell)$ of the contracted matrix. The final matrix chain multiplication employs standard sparse algorithms \cite{davis2006direct,gustavson1978two}.

\subsection{Complete Implementation Framework}

We combine these optimizations into a comprehensive tensor train distributional representation:

\begin{definition}[Complete Tensor Train Implementation]
\label{def:tt_implementation}
Let \textit{DistRegistry} be the distribution type registry from Section~\ref{subsec:registry}, \textit{SparseCores} be the collection $\{\mathcal{G}^{(k)}_{\text{sparse}}\}_{k=1}^d$ of sparse core representations from Section~\ref{subsec:sparsecore}, \textit{OpType} be a flag tracking the arithmetic heritage from Section~\ref{subsec:operations}
and \text{CachedContractions} be precomputed values of $\langle \mathbf{s}_j^{\alpha}, \mathbf{m}_j \rangle$ for each distribution type $j$ and required powers $\alpha$ from Section~\ref{subsec:moments}. Then, a complete tensor train representation consists of:
\begin{equation*}
\text{TTImplementation} = \{\text{DistRegistry}, \text{SparseCores}, \text{OpType}, \text{CachedContractions}\}.
\end{equation*}
\end{definition}
The memory complexity of this implementation is discussed in detail in Section~\ref{sec:complexity}.

The full algorithm for computing the result of the arithmetic operation $\star \in \{+, \times\}$ between two random variables $U, V$ represented as tensor trains is presented in Algorithm~\ref{alg:tt_arithmetic}. For the remainder of the paper we will refer to the resulting algorithm as TT-TTR. The full algorithm for moment computation (Section S2) and a detailed example demonstrating the handling of correlated variables (Section S3) are provided in the Supplementary Materials. Both algorithms illustrate how we take full advantage of the representation in Definition~\ref{def:tt_implementation}. 

\begin{algorithm}
\footnotesize
\caption{TT-TTR with Operation History Tracking}
\label{alg:tt_arithmetic}
\begin{algorithmic}[1]
\REQUIRE Two variables $U, V$ represented as tensor trains, operation $\star \in \{\texttt{ADD}, \texttt{MULT}\}$
\ENSURE Tensor train for $Z = U \star V$
\STATE \textbf{Operation Phase:}
\STATE
\IF{OpHistory($U$) = $\star$ and OpHistory($V$) = $\star$}
    \STATE $\mathcal{A}_{\cap} \gets \text{ancestors}(U) \cap \text{ancestors}(V)$
    \STATE
    \FORALL{$k \in \mathcal{A}_{\cap}$} 
        \IF{$\star = \texttt{ADD}$} 
        \STATE Add coefficients of non-zero entries in core $k$ : $c_U + c_V$
        \ELSE 
        \STATE Multiply coeffs and sum powers of unique entry in core $k$: $c_U \cdot c_V$, $\alpha_U + \alpha_V$
        \ENDIF
    \ENDFOR
    \STATE
    \FORALL{$k \in \text{ancestors}(U) \triangle \text{ancestors}(V)$} 
        \STATE Add core with rank 2 if $\star = \texttt{ADD}$, rank 1 if $\star = \texttt{MULT}$ per Theorems \ref{thm:tt_addition} \& \ref{thm:tt_multiplication}
    \ENDFOR
    \STATE OpHistory($Z$) $\gets \star$
\ELSE
    \STATE Align $U$ and $V$ to common ancestor space (Theorem~\ref{thm:tt_alignment})
    \STATE Apply general tensor train $\star$ operation (Theorem~\ref{thm:tt_general})
    \STATE OpHistory($Z$) $\gets$ \texttt{MIXED}
\ENDIF
\STATE
\STATE Store result using sparse representation: $(i,\ell,\alpha,c)$ tuples
\RETURN tensor train for $Z$
\end{algorithmic}
\end{algorithm}

\section{Computational Complexity Analysis}
\label{sec:complexity}

We analyze the computational complexity of three approaches for handling arithmetic operations on discrete probability distributions: the jTTR approach from Bilgin et al.~\cite{bilgin_patent}, Monte Carlo sampling, and our tensor train TTR (TT-TTR) method. Each approach must handle $d$ independent random variables discretized with $n$ points, where these variables are drawn from $k \leq d$ distinct distributions. All approaches must handle both the independent input variables and the resulting joint distributions from arithmetic operations. For the TT-TTR analysis, we focus on the structural sparsity of tensor train cores, specifically the number of nonzero entries (\textit{nnz}) in lateral slices. Due to our tensor train construction, a zero entry in a lateral slice implies the entire fiber is zero, directly translating to computational savings during contractions.

\begin{remark}[Sparsity preservation in tensor train arithmetic]
\label{rem:nnz}
For cores $\mathcal{X}^{(k)}$ and $\mathcal{Y}^{(k)}$ with $\textit{nnz}_x = r_x^2 - z_x$ and $\textit{nnz}_y = r_y^2 - z_y$ nonzero entries, arithmetic operations preserve sparsity as follows. For addition via Theorem~\ref{thm:tt_general}, the block diagonal structure yields $\textit{nnz}_{\oplus} = \textit{nnz}_x + \textit{nnz}_y$, with $2r_xr_y$ zeros introduced in off-diagonal blocks. For multiplication via Theorem~\ref{thm:tt_general}, the Kronecker product structure yields $\textit{nnz}_{\otimes} = \textit{nnz}_x \cdot \textit{nnz}_y = (r_x^2 - z_x)(r_y^2 - z_y)$. Pure operations achieve optimal sparsity: $\textit{nnz} = 3$ for addition and $\textit{nnz} = 1$ for multiplication.
\end{remark}

\subsection{Complexity Results}

We compare the jTTR approach~\cite{bilgin_patent}, Monte Carlo sampling, and our TT-TTR method. The jTTR approach discretizes each dimension with $n = O(1/\varepsilon)$ points to achieve approximation error $\varepsilon$ per marginal distribution, resulting in $N = n^d = O(\varepsilon^{-d})$ total support points~\cite{bilgin_quantization_2025}. Monte Carlo achieves error $\varepsilon$ in computing expectations with $s = O(1/\varepsilon^2)$ samples, independent of dimension~\cite{glasserman2003monte}. It is important to note that in practice, when relying on finite samples, a non-asymptotic analysis~\cite{MR4777182} shows a clear dependence on the dimension. As a matter of fact, the number of samples required to achieve a good enough approximation might be exponential in the dimension. Our TT-TTR method uses the same $n = O(1/\varepsilon)$ points per dimension as jTTR but avoids constructing the exponentially large joint tensor.

The jTTR approach stores the full joint distribution tensor, requiring $O(\varepsilon^{-d})$ memory, exponential in dimension. Monte Carlo maintains only running statistics, requiring $O(\varepsilon^{-2})$ memory regardless of $d$. Our TT-TTR representation stores $d$ sparse cores plus $k$ distribution types, yielding $O(d \cdot \textit{nnz} + k/\varepsilon)$ memory that scales linearly in $d$ rather than exponentially.

Table~\ref{tab:complexity_comparison} summarizes the computational costs to achieve accuracy $\varepsilon$. For construction, the jTTR approach evaluates the arithmetic expression at all $\varepsilon^{-d}$ outcome combinations, where each evaluation requires $O(d)$ operations on the $d$-tuple, yielding $O(d\varepsilon^{-d})$ total complexity. Monte Carlo draws one value from each of the $d$ input distributions and evaluates the arithmetic expression on the resulting $d$-tuple, repeating this $\varepsilon^{-2}$ times for $O(d\varepsilon^{-2})$ complexity. Our TT-TTR method performs sparse operations on cores with \textit{nnz} nonzero entries, yielding $O(d \cdot \textit{nnz})$ construction complexity. This complexity is independent of $\varepsilon$ once the input distributions are discretized. For moment computation, the TT-TTR approach leverages the $k$ distinct distributions by caching contractions for each distribution type and power, then performs sparse matrix-vector products across cores. The $m$-th moment requires $O(mkp/\varepsilon + d \cdot \textit{nnz}^m)$ operations, where the dependence on $k$ rather than $d$ in the first term provides substantial savings when variables share distributions.

\begin{table}[htbp]
\footnotesize
\centering
\caption{Computational complexity to achieve accuracy $\varepsilon$ per dimension for $d$ variables drawn from $k \leq d$ distinct distributions}
\label{tab:complexity_comparison}
\begin{tabular}{lccc}
\toprule
\textbf{Operation} & \textbf{jTTR~\cite{bilgin_patent}} & \textbf{Monte Carlo}$^\ast$& \textbf{TT-TTR} \\
\midrule
Construction & $O(d\varepsilon^{-d})$ & $O(d\varepsilon^{-2})$ & $O(d \cdot \textit{nnz})$ \\
Storage & $O(\varepsilon^{-d})$ & $O(\varepsilon^{-2})$ & $O(d \cdot \textit{nnz} + k/\varepsilon)$ \\
$m$-th moment & $O(\varepsilon^{-d})$ & $O(\varepsilon^{-2})$ & $O(mkp/\varepsilon + d \cdot \textit{nnz}^m)$ \\
Covariance & $O(\varepsilon^{-2d})^{\dagger}$ & $O(\varepsilon^{-2})$ & $O(kp/\varepsilon + d \cdot \textit{nnz}^2)$ \\
\bottomrule
\multicolumn{4}{l}{\footnotesize $k$ = number of unique distributions, $p$ = number of distinct powers, $\textit{nnz}$ = nonzeros per slice}\\
\multicolumn{4}{l}{\footnotesize $^{\dagger}$Worst case when variables have disjoint ancestors}\\
\multicolumn{4}{l}{\footnotesize $^{\ast}$Non-asymptotic analysis reveals dimension-dependent constants in finite-sample settings~\cite{MR4777182}.}\\
\end{tabular}
\end{table}

The practical impact is substantial. Consider the dice example from Section~\ref{sec:example}: computing statistics for the sum of 1000 dice requires $6^{1000}$ storage for the jTTR approach. In contrast, our TT-TTR method with $\textit{nnz} = 3$ (pure addition) and $k = 1$ (all dice share the same distribution) requires only $d \cdot \textit{nnz} \cdot 4 + k \cdot n \cdot 2 = 1000 \cdot 3 \cdot 4 + 1 \cdot 6 \cdot 2 = 12,012$ doubles.

\section{Numerical Experiments}\label{sec:experiments}
We evaluate the computational performance of Algorithms~\ref{alg:tt_arithmetic} and Algorithm 1 (Section S2, supplementary materials) on benchmark problems from uncertainty quantification and stochastic modeling.

\subsection{Experimental Setup}
We compare our TT-TTR approach against two baselines: (i) the jTTR method~\cite{bilgin_patent} that explicitly constructs all $n^d$ outcome combinations, and (ii) Monte Carlo simulation~\cite{glasserman2003monte, metropolis1949monte} for uncertainty propagation. We exclude existing tensor train uncertainty quantification methods~\cite{dolgov2020approximation, novikov2021tensor} from comparison as they address a fundamentally different problem---constructing surrogates for known probability density functions rather than performing arithmetic operations on discrete representations of arbitrary distributions.

We conduct all experiments using single-threaded execution to isolate algorithmic performance from parallelization effects. We implement all methods in Python with NumPy~\cite{harris2020array} on a MacBook Pro (Apple M3 processor, 8GB RAM).

\subsection{Basic Operations and Statistical Calculations}
We evaluate TT-TTR performance on elementary operations with increasing dimensionality. For $d$ independent gaussian random variables $X_i$, with mean 0 and variance 1, we compute the sum $Z = \sum_{i=1}^d X_i$ and product $W = \prod_{i=1}^d X_i$ for $d = 1, \ldots, 5$, using discretization size $n = 2^5 = 32$ per dimension.

Figure~\ref{fig:perf_basic} shows computation time and memory usage comparisons. The jTTR approach exhibits exponential growth in both metrics, while our TT-TTR implementation scales near-linearly. At $d=5$, TT-TTR operations achieve approximately $10^4\times$ speedup, with addition operations specifically reaching $96,000\times$ faster execution. Memory requirements differ by orders of magnitude: hundreds of MB (jTTR) versus KB (TT-TTR) at $d=5$. Note that this experiment was performed without any of the memory-efficient additions from Section~\ref{sec:algorithm}.

\begin{figure}[ht]
\centering
\includegraphics[width=\textwidth]{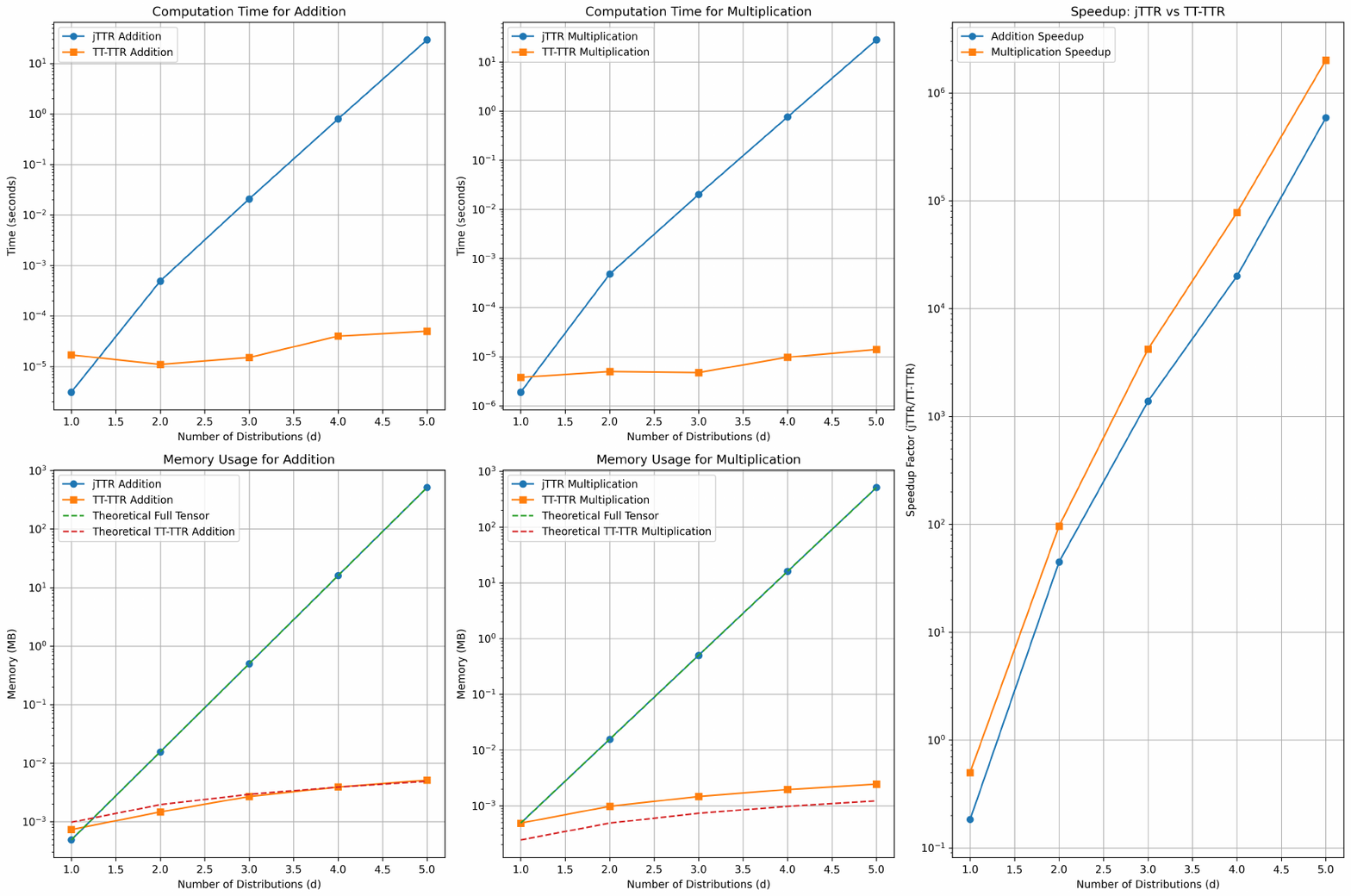}
\caption{Performance comparison between jTTR and TT-TTR implementations. Top row: Computation time for addition (left) and multiplication (right). Bottom row: Memory usage for addition (left) and multiplication (right). Right column: Speedup factor as dimensions increase.}
\label{fig:perf_basic}
\end{figure}

Table~\ref{tab:moments_covariance_combined} compares performance for computing first and second moments $\mu^{\text{disc}}_k$ (Theorem~\ref{thm:moments}). The TT-TTR approach maintains numerical precision while achieving $1,000\times$ speedup at $d=5$. For covariance computation, we evaluate $\text{Cov}^{\text{disc}}(T, Q)$ where $T = X_1 + X_2 + X_3 + X_4$ and $Q = X_5 \times X_6 \times X_7 \times X_8$ with $d=4$ and $n=64$. Table~\ref{tab:moments_covariance_combined} demonstrates accurate covariance computation (Theorem~\ref{thm:covariance}) with similar performance gains. Clearly, since $T$ and $Q$ are uncorrelated, we expect the covariance value to be zero, where the computed value $0.000399$ for both cases is a result of the discretization error.

\begin{table}[ht]
\footnotesize
\setlength{\tabcolsep}{2pt}
\centering
\caption{Performance comparison for moment and covariance computations}
\label{tab:moments_covariance_combined}
\begin{tabular}{l|cc|cc|cc|cc}
\toprule
\multicolumn{9}{c}{\textbf{(a) Moment Computation Performance}} \\
\midrule
$d$ & \multicolumn{2}{c|}{Add (s)} & \multicolumn{2}{c|}{Moments (s)} & \multicolumn{2}{c|}{Speedup} & \multicolumn{2}{c}{Abs. Error} \\
 & jTTR & TT-TTR & jTTR & TT-TTR & Add & Mom & Mean & Var \\
\midrule
2 & $5.0{\times}10^{-4}$ & $1.3{\times}10^{-5}$ & $<\!10^{-4}$ & $5.0{\times}10^{-4}$ & 37 & 0.03 & 0.0 & $8.0{\times}10^{-16}$ \\
3 & $2.1{\times}10^{-2}$ & $4.1{\times}10^{-5}$ & $1.3{\times}10^{-3}$ & $2.4{\times}10^{-4}$ & 512 & 5.4 & $1.4{\times}10^{-16}$ & $3.2{\times}10^{-15}$ \\
4 & $7.9{\times}10^{-1}$ & $7.9{\times}10^{-5}$ & $1.4{\times}10^{-2}$ & $3.5{\times}10^{-4}$ & $10^4$ & 39 & $6.0{\times}10^{-16}$ & $1.1{\times}10^{-15}$ \\
5 & $2.9{\times}10^{1}$ & $3.0{\times}10^{-4}$ & $6.8{\times}10^{-1}$ & $6.4{\times}10^{-4}$ & $9.6{\times}10^4$ & $10^3$ & $3.4{\times}10^{-15}$ & $3.2{\times}10^{-15}$ \\
\midrule
\multicolumn{9}{c}{\textbf{(b) Covariance Computation ($n=64$, $d=4$)}} \\
\midrule
\multicolumn{3}{l|}{Metric} & \multicolumn{2}{c|}{jTTR} & \multicolumn{2}{c|}{TT-TTR} & \multicolumn{2}{c}{Speedup} \\
\midrule
\multicolumn{3}{l|}{Construction (s)} & \multicolumn{2}{c|}{24.49} & \multicolumn{2}{c|}{0.0003} & \multicolumn{2}{c}{80,256×} \\
\multicolumn{3}{l|}{Calculation (s)} & \multicolumn{2}{c|}{0.46} & \multicolumn{2}{c|}{0.0017} & \multicolumn{2}{c}{273×} \\
\multicolumn{3}{l|}{Total (s)} & \multicolumn{2}{c|}{24.95} & \multicolumn{2}{c|}{0.0020} & \multicolumn{2}{c}{12,550×} \\
\multicolumn{3}{l|}{Memory (MB)} & \multicolumn{2}{c|}{384} & \multicolumn{2}{c|}{0.01} & \multicolumn{2}{c}{32,768×} \\
\multicolumn{3}{l|}{Cov. Value} & \multicolumn{2}{c|}{0.000399} & \multicolumn{2}{c|}{0.000399} & \multicolumn{2}{c}{---} \\
\multicolumn{9}{l}{\footnotesize TT-TTR Ranks: $T$ [1,2,2,2,1], $Q$ [1,1,1,1,1]; Relative Error: $2.55 \times 10^{-13}$} \\
\bottomrule
\end{tabular}
\end{table}

These results confirm the complexity predicted in Section~\ref{sec:complexity} where jTTR enumeration with $n^d$ outcomes becomes infeasible, with performance advantages growing exponentially with dimension.

\subsection{Problem 2: Monte Carlo Integration}
This benchmark applies our approach to high-dimensional integration problems of the form $\int_\mathcal{D} f(X_1, \ldots, X_d) d\mathbf{X}$. Such integrals are ubiquitous in Bayesian inference \cite{robert2007bayesian}, rare event simulation \cite{rubino2009rare}, uncertainty quantification \cite{ghanem2017handbook} and further domains ranging from molecular energy estimation in chemistry, to derivatives pricing in finance \cite{glasserman2003monte, ceperley1980ground, neal1993probabilistic}. Traditional deterministic methods, such as Riemann sums or Gaussian quadrature, often struggle with such functions, due to their computational inefficiency and sensitivity to dimensionality.

We compare our approach against standard Monte Carlo integration \cite{metropolis1949monte, glasserman2003monte} to demonstrate our method's advantages. While more sophisticated variants exist, such as quasi-Monte Carlo \cite{dick2013high}, stratified sampling \cite{glynn1992asymptotic}, and importance sampling \cite{owen2013monte}, these methods still suffer from sampling variance and often require problem-specific tuning. Our deterministic approach eliminates sampling variance entirely, achieving machine precision regardless of integrand properties for the considered problems.

We first consider the $d$-dimensional integral
\begin{equation}
I = \int_{[\ell,u]^d} \prod_{i=1}^{d} x_i \, d\mathbf{x},
\end{equation}
which admits the analytical solution $I = \prod_{j=1}^{d} (u^2 - \ell^2)/2$. The Monte Carlo estimator approximates $I$ using $s$ uniformly distributed sample points $\mathbf{x}_k \in [\ell,u]^d$: $\hat{I}_{\text{MC}} = (u - \ell)^d s^{-1} \sum_{k=1}^{s} f(\mathbf{x}_k)$. The TT-TTR approach takes as input $d$ uniform random variables $X_k \in [\ell, u]$ represented by $n$ TTR discretization points. The integral is then computed as $\hat{I}_{\text{TT-TTR}} = (u - \ell)^d \mathbb{E}[\prod_{k=1}^d X_k]$ using Algorithms~\ref{alg:tt_arithmetic} and Algorithm 1 (Section S2, supplementary materials).

We test dimensions $d \in \{5, 25, 45\}$ with sample sizes $s \in \{10^5, 10^6\}$ and representation sizes $n \in \{16, 64\}$. Each Monte Carlo configuration undergoes 50 independent runs to compute the mean relative error $\epsilon_{\text{rel}} = |\hat{I}_{\text{MC}} - I| / I$. We use $u = 2$ throughout and vary $\ell \in \{-1, 1\}$ to illustrate Monte Carlo's sensitivity to zeros in the integrand.

For $f(\mathbf{x}) = \prod_{i=1}^d x_i$ on $[-1,2]^d$, the mean $\mathbb{E}[f] = (0.5)^d$ while $\mathbb{E}[f^2] = 1$, yielding relative variance $\text{Var}[f]/\mathbb{E}[f]^2 \approx 4^d$. This exponential growth causes the catastrophic Monte Carlo errors shown in Table~\ref{tab:mc_tt_comparison}. The domain $[1,2]^d$ avoids zeros, restoring standard convergence. Table~\ref{tab:mc_tt_comparison} shows the TT-TTR approach achieves machine precision for both domains in orders of magnitude less time. This is due to the mean-preserving property of the TTR discretization.

\begin{table}[htbp]
\footnotesize
\setlength{\tabcolsep}{3pt}
\centering
\caption{Monte Carlo vs TT-TTR integration for $\int \prod_{i=1}^{d} x_i \, d\mathbf{x}$ over different domains}
\label{tab:mc_tt_comparison}
\begin{tabular}{@{}r|cc|cc|cc|cc@{}}
\toprule
& \multicolumn{4}{c|}{$[1,2]^d$ (no zeros)} & \multicolumn{4}{c}{$[-1,2]^d$ (with zeros)} \\
\cmidrule{2-5} \cmidrule{6-9}
& \multicolumn{2}{c|}{MC} & \multicolumn{2}{c|}{TT-TTR} & \multicolumn{2}{c|}{MC} & \multicolumn{2}{c}{TT-TTR} \\
$d$ & Err (\%) & Time & Err & Time & Err (\%) & Time & Err & Time \\
\midrule
\multicolumn{9}{c}{\textit{MC: $s = 10^5$, TT-TTR: $n = 16$}} \\
\midrule
 5 & 0.06 & .005 & $<\!10^{-15}$ & .0001 & $8.8$ & .005 & $<\!10^{-15}$ & .0001 \\
25 & 0.32 & .020 & $<\!10^{-15}$ & .0006 & $2.7{\times}10^{6}$ & .023 & $10^{-14.2}$ & .0008 \\
45 & 0.49 & .037 & $10^{-14.8}$ & .0016 & $2.2{\times}10^{12}$ & .039 & $10^{-13.9}$ & .0014 \\
\midrule
\multicolumn{9}{c}{\textit{MC: $s = 10^6$, TT-TTR: $n = 64$}} \\
\midrule
 5 & 0.04 & .049 & $<\!10^{-15}$ & .0001 & $3.1$ & .048 & $10^{-14.9}$ & .0001 \\
25 & 0.10 & .226 & $<\!10^{-15}$ & .0005 & $2.2{\times}10^{6}$ & .238 & $10^{-14.3}$ & .0005 \\
45 & 0.16 & .400 & $10^{-14.8}$ & .0014 & $8.4{\times}10^{12}$ & .439 & $10^{-14.1}$ & .0024 \\
\bottomrule
\end{tabular}
\end{table}

We next test correlation handling with $\mathbf{y} = \mathbf{A}\mathbf{x}$ where $\mathbf{A} \in \mathbb{R}^{d \times d}$ is orthogonal and $\mathbf{x} \sim \text{Uniform}[0,1]^d$. We compute $\int_{\mathbf{A}[0,1]^d} \sum_{i=1}^d y_i \, d\mathbf{y}$ with solution $|\det(\mathbf{A})| \sum_{i,j} \mathbf{A}_{ij}/2$. Table~\ref{tab:transform_comparison} shows TT-TTR integration achieves machine precision for all dimensions and representation sizes while Monte Carlo struggles to reach 0.1\% accuracy even with $10^6$ samples.

\begin{table}[htbp]
\footnotesize
\centering
\caption{Integration performance for $\int_{\mathbf{A}[0,1]^d} \sum_i y_i \, d\mathbf{y}$ with random orthogonal $\mathbf{A}$}
\label{tab:transform_comparison}
\begin{tabular}{@{}r|rrr|rrr@{}}
\toprule
& \multicolumn{3}{c|}{Monte Carlo} & \multicolumn{3}{c}{TT-TTR} \\
$d$ & $s$ & Error (\%) & Time (s) & $n$ & Error & Time (s) \\
\midrule
\multirow{3}{*}{40} & $10^4$ & 5.02 & 0.004 & 8 & $<10^{-15}$ & 0.009 \\
                     & $10^5$ & 1.91 & 0.044 & 16 & $<10^{-15}$ & 0.008 \\
                     & $10^6$ & 0.09 & 0.372 & 32 & $<10^{-15}$ & 0.008 \\
\midrule
\multirow{3}{*}{120} & $10^4$ & 4.01 & 0.010 & 8 & $<10^{-15}$ & 0.140 \\
                      & $10^5$ & 1.54 & 0.102 & 16 & $<10^{-15}$ & 0.159 \\
                      & $10^6$ & 0.36 & 1.589 & 32 & $<10^{-15}$ & 0.141 \\
\midrule
\multirow{3}{*}{200} & $10^4$ & 0.51 & 0.019 & 8 & $<10^{-15}$ & 0.598 \\
                      & $10^5$ & 0.65 & 0.179 & 16 & $<10^{-15}$ & 0.572 \\
                      & $10^6$ & 0.25 & 3.312 & 32 & $<10^{-15}$ & 0.678 \\
\bottomrule
\end{tabular}
\end{table}

\subsection{Problem 3: Integrated Brownian Motion}
This benchmark implements a numerical solution of the stochastic differential equation (SDE) for an integrated Brownian motion (IBM) process \cite{oosterlee2019mathematical}. We compute not only the Wiener process, or Brownian motion $W_T$, but also the integral quantities $I_1 = \int_0^T W_s ds$ and $I_2 = \int_0^T W_s dW_s$ using numerical integration. 

\begin{figure}[ht]
\centering
\includegraphics[width=\textwidth]{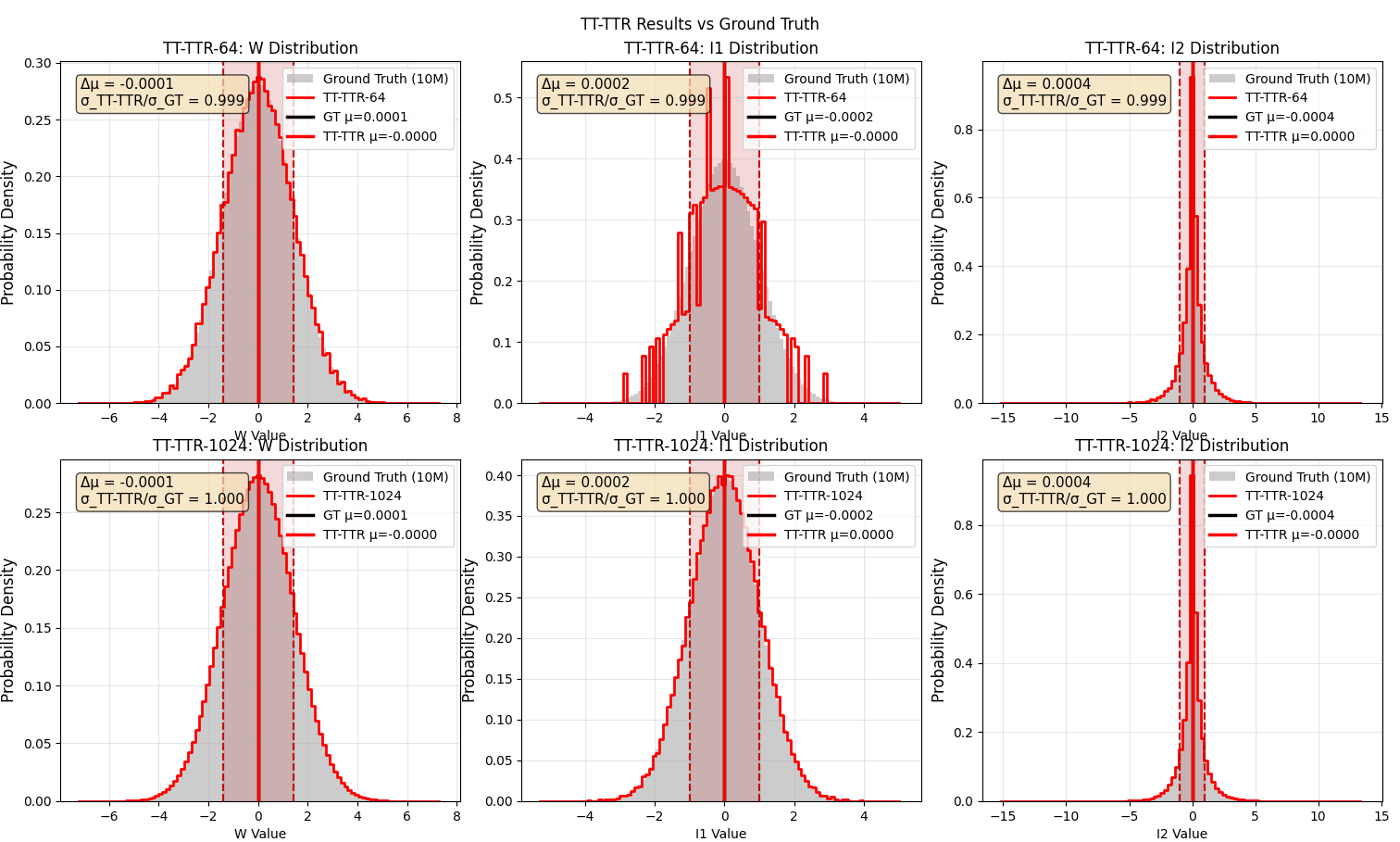}
\caption{Probability density functions of $W_T$, $I_1$, and $I_2$ from the TT-TTR approach with $n=64$ and $n=1024$. Histograms are constructed by binning the discrete support points and their masses.}
\label{fig:tthist}
\end{figure}

These three quantities form our output vector, providing a comprehensive test of correlation handling in the TT-TTR framework. These quantities are relevant in Asian option pricing \cite{vecer2001pricing}, where the payoff depends on the average asset price over time. They also arise in control theory for integrating noise processes \cite{oksendal2013stochastic} and in statistical physics for computing path integrals \cite{kleinert2009path}.

We compare the performance of Algorithm~\ref{alg:tt_arithmetic} against standard Monte Carlo path integrals for this problem. To validate our implementation, we first integrate both the TT-TTR and Monte Carlo approaches to $T = 2$ and compare the accuracy of the resulting probability density functions (PDFs). We create a ground truth PDF by simulating $10^7$ Monte Carlo samples. To compare the accuracy of the different approaches, we plot the PDFs for all three outputs from TT-TTR with $n = 64$ and $n = 1024$ (Figure~\ref{fig:tthist}) and from Monte Carlo with $s = 1000$ and $s = 500000$ (Figure~\ref{fig:mchist}). On these plots, we also compare the computed means and standard deviations. The TT-TTR approach computes the exact mean of 0 up to machine precision, which aligns with the results of the mean-based methods. The small difference in the TT-TTR mean in the figure is, in fact, due to sampling error in the ground truth Monte Carlo.\footnote{For visualization, we construct the histogram for the TT-TTR approach by binning the discrete support points and their corresponding masses, which we obtain by enumerating all the tensor indices, into the same bins as the ground truth histogram.}

\begin{figure}[ht]
\centering
\includegraphics[width=\textwidth]{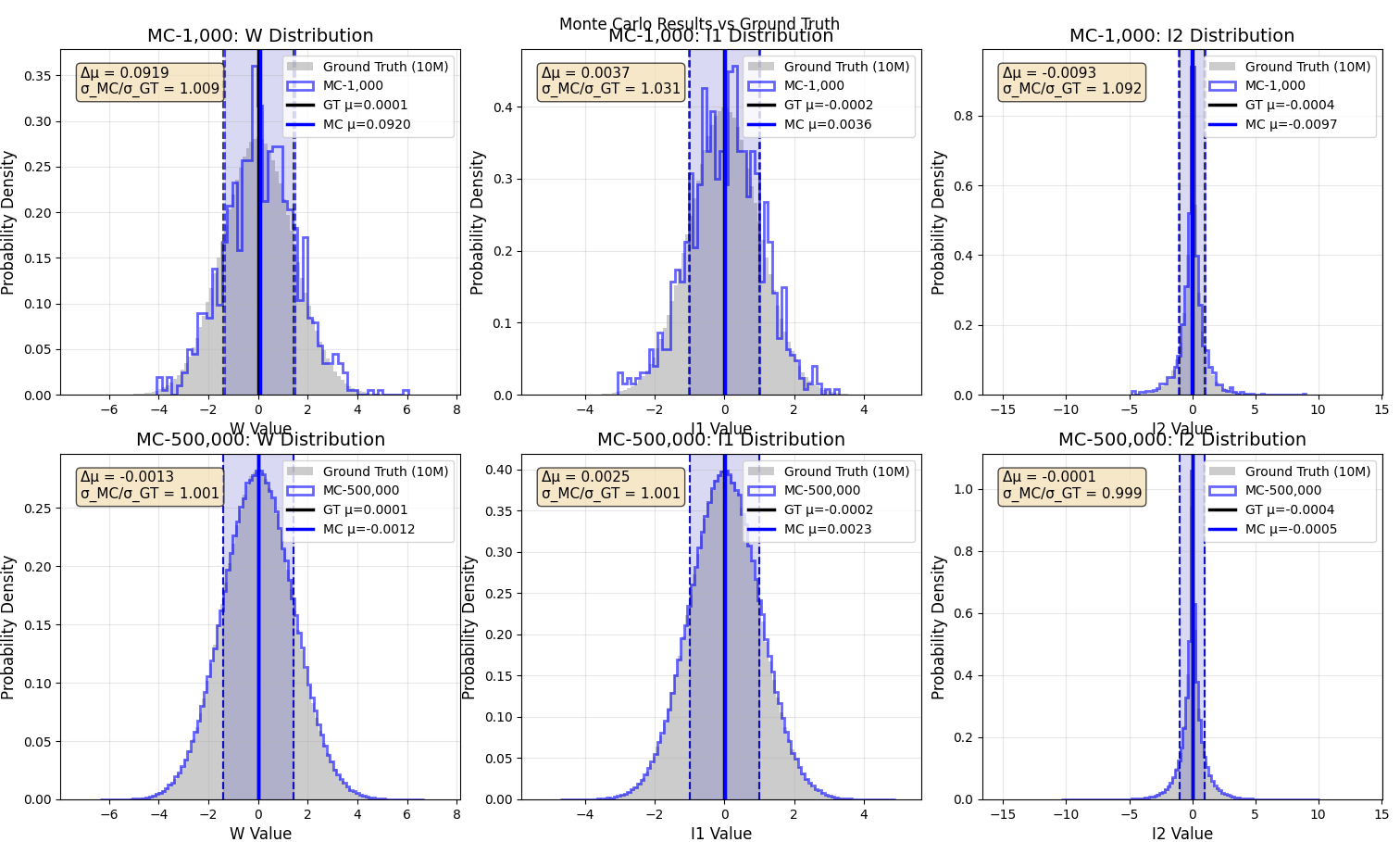}
\caption{Probability density functions of the three outputs from Monte Carlo with $s=1000$ and $s=500000$ paths.}
\label{fig:mchist}
\end{figure}

These plots provide insight into the typical convergence of the TT-TTR and Monte Carlo methods. Building on this understanding, we compare the results on a larger scale by performing the numerical integration up to $T=126$. This corresponds to six months of daily trading day timesteps if the process were based on a financial instrument. Table~\ref{tab:simple_comparison} presents the CPU time and memory comparison between TT-TTR and Monte Carlo for integrating the full path with different parameters and computing the expected value of the three output variables.

\begin{table}[h!]
\footnotesize
\centering
\caption{Performance comparison of TT-TTR and Monte Carlo methods for the Integrated Brownian Motion Benchmark}
\label{tab:simple_comparison}
\begin{tabular}{@{}lcc|lcc@{}}
\toprule
\multicolumn{3}{c|}{TT-TTR} & \multicolumn{3}{c}{Monte Carlo} \\
$n$ & Time (s) & Mem (MB) & $s$ & Time (s) & Mem (MB) \\
\midrule
128 & 0.320 & 54.4 & 100k & 0.161 & 5.3 \\
512 & 0.341 & 59.4 & 1M & 1.799 & 53.4 \\
2048 & 0.437 & 64.7 & 5M & 9.188 & 267.0 \\
\bottomrule
\end{tabular}
\end{table}

This experiment clearly demonstrates the power of our sparse data representation discussed in Section~\ref{sec:algorithm}. The CPU time and memory requirements for the TT-TTR approach remain nearly constant as we drastically refine the discretizations, while Monte Carlo shows linear growth in both metrics with the number of samples $s$. Notably, achieving comparable accuracy would require orders of magnitude more Monte Carlo samples, making the TT-TTR approach particularly advantageous for high-precision requirements.

\subsection{Problem 4: Hutchinson's Trace Estimator}
We evaluate our framework on Hutchinson's trace estimator \cite{hutchinson1990stochastic}, which approximates $\text{tr}(\mathbf{A}) = \mathbb{E}[\mathbf{v}^T \mathbf{A} \mathbf{v}]$ for $\mathbf{v}$ with i.i.d. Rademacher entries. Table~\ref{tab:trace_estimation} compares direct computation, Monte Carlo sampling, and our TT-TTR approach for $d \times d$ SPD matrices with eigenvalues in $[0.1, 1]$.

\begin{table}[htbp]
\footnotesize
\centering
\caption{Trace estimation performance across methods and dimensions}
\label{tab:trace_estimation}
\begin{tabular}{@{}lc|ccc|ccc|ccc@{}}
\toprule
& & \multicolumn{3}{c|}{$d = 40$} & \multicolumn{3}{c|}{$d = 120$} & \multicolumn{3}{c}{$d = 200$} \\
Method & $s/n$ & Est & Err & Time & Est & Err & Time & Est & Err & Time \\
\midrule
Direct & -- & 20.366 & 0 & .000 & 67.133 & 0 & .000 & 111.35 & 0 & .001 \\
MC & 100 & 20.536 & .170 & .001 & 67.426 & .293 & .001 & 111.84 & .493 & .001 \\
MC & 10k & 20.383 & .017 & .076 & 67.171 & .038 & .080 & 111.39 & .048 & .090 \\
TT-TTR & 2 & 20.366 & 0 & .015 & 67.133 & $<$1e-15 & .436 & 111.35 & $<$1e-15 & .984 \\
\bottomrule
\end{tabular}
\end{table}

Although our TT-TTR approach achieves machine precision, it requires more computation time than direct evaluation or Monte Carlo for this problem. This example validates the correctness of our TT-TTR arithmetic framework while highlighting opportunities for improvement. Future work incorporating controlled approximations within the TT-TTR structure could reduce computational costs while preserving deterministic accuracy, making the approach competitive for applications requiring high-precision variance reduction.

\section{Conclusions}\label{sec:conclusions}

We developed a tensor train TTR framework (TT-TTR) for exact arithmetic operations on discrete representations of probability distributions. The approach combines Dirac mixture representations with tensor train decompositions to achieve polynomial complexity scaling while maintaining deterministic computation. This addresses the exponential scaling limitation that prevents existing deterministic methods such as jTTR~\cite{bilgin_patent} from handling high-dimensional problems.

The key contributions of this work include algorithms that adaptively preserve the tensor train structure under addition, subtraction, and multiplication of random variables. We proved polynomial complexity bounds under standard rank assumptions and implemented the method using sparse data structures that further reduce memory requirements by a factor proportional to the discretization size. The implementation computes moments, and covariances of latent variables with polynomial complexity.

Experiments validated the theoretical predictions across multiple domains. For basic arithmetic operations, we achieved speedups exceeding $10^4$ at dimension $d=5$. Monte Carlo integration tests showed our method maintaining machine precision where sampling approaches exhibited errors of $10^{12}$ due to high variance integrands. The 1,000-dice example reduced memory requirements from $6^{1000}$ to 12,012 doubles. The results for Integrated Brownian Motion demonstrate particular promise for challenging stochastic calculus applications, which are crucial in financial engineering for modeling path-dependent derivatives and credit risk, and traditionally requires millions of Monte Carlo samples for accurate computation. For trace estimation, while our current results do not match direct computation speed, the experiments demonstrate new avenues for tensor train applications. The ability to approximate traces using only 2 discretization points opens possibilities for developing specialized intra-tensor-train approximation methods that could exploit the compressed representation structure for randomized linear algebra operations.

\subsection{Limitations of TT-TTR}
While our results are promising, important challenges remain.

The framework excels at linear operations like addition and multiplication, but many applications demand more. Division, exponentiation, and other nonlinear operations break the tensor train structure.

Statistical extraction presents another frontier. Computing moments and covariances is straightforward (Theorems~\ref{thm:moments} and~\ref{thm:covariance}), but financial applications often need quantiles for risk assessment \cite{hull2018options}. Extracting quantiles from tensor train representations remains challenging since it requires efficient CDF evaluation.

Despite these limitations our method enables previously intractable deterministic computations in uncertainty quantification, as demonstrated in Section~\ref{sec:experiments}. The ability to perform deterministic operations on probability distributions opens new possibilities in uncertainty quantification, particularly where reproducibility and exactness are paramount. As tensor decomposition methods advance, we anticipate that many current limitations will be addressed, expanding the applicability of distributional arithmetic.

\bibliographystyle{siam}
\bibliography{references}

\end{document}